\documentclass[11pt]{amsart}

\usepackage{amsmath,amssymb,amsthm,mathtools,mathrsfs}
\usepackage{comment}
\usepackage[margin=1.15in]{geometry}
\usepackage{enumitem}
\usepackage{esint}
\usepackage[colorlinks=true,citecolor=blue,linkcolor=blue,urlcolor=blue]{hyperref}
\DeclareRobustCommand{\SkipTocEntry}[5]{}
\newtheorem{theorem}{Theorem}[section]
\newtheorem{proposition}[theorem]{Proposition}
\newtheorem{lemma}[theorem]{Lemma}
\newtheorem{corollary}[theorem]{Corollary}
\newtheorem{definition}[theorem]{Definition}

\newtheorem{question}[theorem]{Question}

\newtheorem{mainresult}{Theorem}

\newtheorem{maincorollary}[mainresult]{Corollary}

\newcommand{\Hpot}{\mathcal H}
\newcommand{\vap}{\mathcal V_\alpha^+}
 
\newcommand{\PSH}{\operatorname{PSH}}
\newcommand{\Ric}{\operatorname{Ric}}
\newcommand{\Scal}{\operatorname{Scal}}
\newcommand{\Ent}{\operatorname{Ent}}

\newcommand{\ddbar}{\sqrt{-1}\,\partial\bar\partial}

\title[A Variational Characterization of PSC K\"ahler metrics]{A Variational Characterization of Positive Scalar Curvature K\"ahler metrics}
\author{Zehao Sha}

\begin{document}

\begin{abstract}
We introduce the prescribed scalar curvature measure equation on a compact K\"ahler manifold. For a K\"ahler class of positive total scalar curvature, we prove that the following are equivalent: the existence of a positive scalar curvature K\"ahler metric, solvability of this equation for every admissible measure, \(d_1\)-coercivity of the associated functionals, and uniform geodesic stability along finite-energy \(d_1\)-geodesic rays. As a consequence, in each fixed K\"ahler class, the space of positive scalar curvature K\"ahler metrics is either empty or contractible. We further prove that every K\"ahler class on a positive-dimensional compact smooth toric K\"ahler manifold contains a torus-invariant metric of positive scalar curvature. Therefore, for every K\"ahler class, the prescribed scalar curvature measure equation admits a smooth solution for every admissible measure, unique modulo constants.
\end{abstract}

\maketitle
\tableofcontents
\section{Introduction}
\label{sec:introduction}
\subsection{Positive scalar curvature K\"ahler metrics}

On a closed manifold, the condition of positive scalar curvature \(\Scal(g)>0\) is open in the \(C^2\)-topology, and the existence of a positive scalar curvature (PSC) metric is preserved under surgeries of codimension at least three \cite{SchoenYau1979Structure,GromovLawson1980}. At the same time, PSC is subject to deep topological obstructions. For example, Schoen--Yau proved that \(T^m\) admits no PSC metric for \(3\leq m\leq 7\), and Gromov--Lawson subsequently established the result for every \(m\geq 3\) by spin methods \cite{SchoenYau1979Structure,GromovLawson1980Spin}. A recurring principle in Gromov's approach to positive scalar curvature is that a uniform positive lower bound for scalar curvature should force a loss of two dimensions at large scales; see \cite{Gromov2023FourLectures}. A quantitative expression of this principle is provided by Gromov--Zhu \cite{GromovZhu2024}.

This codimension-two viewpoint has a suggestive counterpart in complex geometry. A rational curve on a complex manifold \(X\) is the image of a nonconstant holomorphic map \(\mathbb P^1\to X\), which is topologically the two-sphere \(S^2\). The manifold \(X\) is called \emph{uniruled} if a general point of \(X\) lies on a rational curve \cite{Kollar1996}. On Hermitian manifolds, Yang proved that a compact complex manifold admits a Hermitian metric with positive Chern scalar curvature if and only if its canonical bundle is not pseudoeffective \cite[Theorem~1.3]{Yang2019}. More recently, Ou established the equivalence between the non-pseudoeffectivity of \(K_X\) and uniruledness for arbitrary compact K\"ahler manifolds \cite[Theorem~1.1]{Ou2025}. Since the Chern and Riemannian scalar curvatures of a K\"ahler metric are proportional, every compact K\"ahler manifold admitting a PSC K\"ahler metric is therefore uniruled.

In complex dimension two, combining Yau's result \cite{Yau1974}, the Enriques--Kodaira classification \cite[Chapter~VI, \S1]{BHPV2004}, and the scalar curvature approximation on point-blowups \cite[Theorem A]{Brown24}, one obtains the following equivalences for every compact complex surface \(X\) of K\"ahler type \cite[Theorem~B]{Brown24}:
\begin{enumerate}
    \item \(X\) admits a K\"ahler metric of positive total scalar curvature;
    \item \(X\) admits a PSC K\"ahler metric;
    \item \(\kappa(X)=-\infty\);
    \item \(X\) is obtained from \(\mathbb P^2\) or a ruled surface
    \(\mathbb P(E)\) by finitely many point blowups, where \(E\) is a
    rank-two holomorphic vector bundle over a compact Riemann surface.
\end{enumerate}

To our knowledge, no classification of compact PSC K\"ahler manifolds is known in complex dimension at least three. 

\subsection{The prescribed scalar curvature measure equation}

The classical prescribed scalar curvature problem asks which smooth functions on a closed manifold can be realized as the scalar curvature of a Riemannian metric. In real dimension at least three, Kazdan and
Warner obtained the well-known trichotomy for closed manifolds \cite{KazdanWarner1975Scalar,KazdanWarner1975Existence}.
A related problem in Hermitian geometry is to prescribe the Chern scalar curvature within a fixed conformal class. The constant-curvature case is the Chern--Yamabe problem studied in \cite{AngellaCalamaiSpotti2017}, while prescribed nonconstant Chern scalar curvature was considered in \cite{Fusi2022,Yu2023}. Such conformal changes do not preserve a fixed K\"ahler class in general.

Here we impose a different constraint: we fix a K\"ahler class \(\alpha\in\mathcal K_X\) and vary the metric only within this class. Let \(\omega \in \alpha\) be a K\"ahler metric. The equation
\[
        \Scal\left(\omega+\ddbar\phi\right)=f
\]
for a prescribed smooth function \(f\) is contained in the general coupled scalar-curvature system studied by Chen--Cheng \cite{ChenCheng2021I,ChenCheng2021II}. Building on the variational theory surrounding the Mabuchi \(K\)-energy \cite{Mabuchi1986,BBGZ2013,BermanBerndtsson2017,DR2017}, we prescribe a positive measure rather than the scalar curvature function. We explain this distinction in Section~\ref{subsec:scalar-measure-equation}.

Suppose that \(\alpha\) is a K\"ahler class with positive total scalar curvature \(S_\alpha:=2n\pi c_1(X) \cdot \alpha^{n-1}\), and define the space of admissible measures by
\[
\vap
:=
\left\{
\Omega\in
C^\infty
\bigl(
X,\Lambda_{\mathbb R}^{n,n}
\bigr):
\Omega>0,~
\int_X\Omega=S_\alpha
\right\}.
\]
Fix a reference K\"ahler metric \(\omega \in \alpha\), and write
\[
\Hpot_\omega:=\left\{\psi \in C^\infty(X,\mathbb R);~\omega_\phi:=\omega+\ddbar \psi >0\right\}.
\]
For \(\Omega\in\mathcal V_\alpha^+\), we consider the \textit{prescribed scalar curvature measure equation}
\begin{equation}
\Scal(\omega_\phi)\,\omega_\phi^n
=
\Omega.
\label{eq:intro-scalar-measure-equation}
\end{equation}
Every solution of \eqref{eq:intro-scalar-measure-equation} determines a K\"ahler metric of positive scalar curvature.

Let \(\mathcal M_\omega\) denote the Mabuchi \(K\)-energy,
\(E_\omega\) the Aubin--Yau energy, and 
\[
L_{\omega,\Omega}(\phi)
:=
\int_X\phi\,\Omega, \qquad \forall \phi \in \Hpot_\omega.
\]
We introduce the prescribed scalar curvature measure functional, which we call the \(\mathcal F\)-functional:
\begin{equation}
\mathcal F_{\omega,\Omega}
:=
\mathcal M_\omega
-
\frac{S_\alpha}{\alpha^n}\, E_\omega
+
L_{\omega,\Omega},
\label{eq:intro-prescribed-functional}
\end{equation}
whose Euler--Lagrange equation is precisely
\eqref{eq:intro-scalar-measure-equation}.

Let \((\mathcal E^1(X,\omega),d_1)\) denote the completion of \(\Hpot_\omega\) in \(d_1\)-distance. The functional
\(\mathcal F_{\omega,\Omega}\) extends to a
\(d_1\)-lower semicontinuous, geodesically convex functional on \(E^1(X,\omega)\); see Section~\ref{subsec:finite-energy-stability}. Since \(\mathcal F_{\omega,\Omega}\) is invariant under addition of constants, we work on \(\mathcal E^1(X,\omega)/\mathbb R\), represented by the canonical slice
\[
        \mathcal E^1_0(X,\omega)
        :=
        \left\{
        \phi\in\mathcal E^1(X,\omega):
        E_\omega(\phi)=0
        \right\},
\]
endowed with the restriction of the \(d_1\)-metric.

For a fixed reference metric \(\omega\in\alpha\) and 
\(\Omega\in\mathcal V_\alpha^+\), we say that
\(\mathcal F_{\omega,\Omega}\) is \(d_1\)-\emph{coercive} if there exist constants \(\delta>0\) and \(C>0\) such that for any \(\phi \in \mathcal E^1_0(X,\omega)\),
\[
        \mathcal F_{\omega,\Omega}(\phi)
        \geq
        \delta\,d_1(0,\phi)-C.
\]
As proved in Section~\ref{subsec:target-independence}, the validity of this condition is independent of the choice of reference metric and admissible measure. We therefore say simply that the \(\mathcal F\)-functional is \(d_1\)-coercive when this condition holds for one, equivalently every, choice of \(\omega\in\alpha\) and \(\Omega\in\mathcal V_\alpha^+\).

Let \(\ell:[0,\infty)\longrightarrow \mathcal E^1_0(X,\omega)\)
be a finite-energy \(d_1\)-geodesic ray with \(\ell_0=0\) parameterized with constant \(d_1\)-speed, meaning that
\[
        d_1(\ell_s,\ell_t)
        =
        |t-s|\,v_1(\ell),
        \qquad
        v_1(\ell):=d_1(\ell_0,\ell_1).
\] 
For \(\Omega\in\mathcal V_\alpha^+\), define the radial slope of \(\mathcal F_{\omega,\Omega}\) by
\[
        \mathcal F_{\omega,\Omega}\{\ell\}
        :=
        \lim_{t\to\infty}
        \frac{\mathcal F_{\omega,\Omega}(\ell_t)}{t}.
\]
The limit exists in \((-\infty,+\infty]\) by convexity along finite-energy geodesics. The radial slope is unchanged when \(\Omega\) is replaced by another
admissible measure and, under the natural identifications of finite-energy spaces and parallel rays, when the reference metric or the initial point is changed. We therefore write
\[
        \mathcal F_\alpha\{\ell\}
        :=
        \mathcal F_{\omega,\Omega}\{\ell\}.
\]
We say that the \(\mathcal F\)-functional is
\emph{uniformly geodesically stable} if there exists a constant \(\delta>0\) such that
\[
        \mathcal F_\alpha\{\ell\}
        \geq
        \delta\,v_1(\ell)
\]
for every nonconstant finite-energy \(d_1\)-geodesic ray starting from \(0\) and parameterized with constant \(d_1\)-speed.

In particular, if \(\alpha\) contains a PSC K\"ahler metric \(\omega_\phi\), then \(\phi\) solves \eqref{eq:intro-scalar-measure-equation} with prescribed measure \(\Omega:=\Scal(\omega_\phi)\,\omega_\phi^n\). We further assert that solvability for one admissible
measure can upgrade to solvability for every admissible measure, and that this property is characterized by coercivity and by the asymptotic behavior of \(\mathcal F\) along finite-energy geodesic rays. Our main analytic result is the following.

\begin{mainresult}
\label{thm:variational-main}
Let \(\alpha\) be a K\"ahler class with positive total scalar curvature, and let \(\omega \in \alpha\) be a fixed reference K\"ahler metric. Then the following conditions are equivalent:
\begin{enumerate}[label=\textup{(\roman*)}]
\item
the K\"ahler class \(\alpha\) contains a K\"ahler metric of positive scalar curvature;

\item
there exists \(\Omega\in\mathcal V_\alpha^+\)
such that equation \eqref{eq:intro-scalar-measure-equation}
admits a smooth solution in \(\Hpot_\omega\);

\item
for every \(\Omega\in\mathcal V_\alpha^+\), equation
\eqref{eq:intro-scalar-measure-equation}
admits a smooth solution in \(\Hpot_\omega\);

\item
the \(\mathcal F\)-functional is \(d_1\)-coercive;

\item
the \(\mathcal F\)-functional is uniformly geodesically stable.
\end{enumerate}
Moreover, for every fixed \(\Omega\in\mathcal V_\alpha^+\), equation \eqref{eq:intro-scalar-measure-equation} has at most one smooth solution up to addition of constants.
\end{mainresult}

Theorem~\ref{thm:variational-main} also gives a global description of the entire PSC locus in a fixed K\"ahler class. Denote by
\[
        \mathcal H_\omega^{\mathrm{psc}}
        :=
        \left\{
        \phi\in\mathcal H_\omega:
        \Scal(\omega_\phi)>0
        \right\}.
\]
Since adding a constant does not change the associated K\"ahler metric,
the scalar curvature measure map naturally descends to
\[
        \mathscr S_\omega:
        \mathcal H_\omega^{\mathrm{psc}}/\mathbb R
        \longrightarrow
        \mathcal V_\alpha^+,
        \qquad
        [\phi]
        \longmapsto
        \Scal(\omega_\phi)\,\omega_\phi^n .
\]

\begin{maincorollary}
\label{cor:global-parametrization-psc}
Let \(X\) be a compact K\"ahler manifold and let \(\alpha\) be a K\"ahler class. Suppose that \(\alpha\) contains a positive scalar curvature K\"ahler metric. Then \(\mathscr S_\omega\) is a homeomorphism with respect to the \(C^\infty\)-topologies. In particular, \(\mathcal H_\omega^{\mathrm{psc}}/\mathbb R\) is contractible.
\end{maincorollary}
Thus the entire PSC locus in a fixed K\"ahler class is globally parametrized by a convex space of positive volume forms with prescribed total mass.

\subsection{Compact toric K\"ahler manifolds}

Toric geometry provides a natural setting in which the existence of PSC K\"ahler metrics can be reduced to a problem in convex analysis. Let \(X\) be a compact smooth toric K\"ahler manifold of complex dimension \(n\geq1\), and fix a K\"ahler class \(\alpha\). Averaging any K\"ahler form in \(\alpha\) over the compact
torus \(T^n\) produces a \(T^n\)-invariant reference K\"ahler form \(\omega\in\alpha\). After choosing a normalization of the corresponding moment map, its image is a Delzant polytope \(\Delta\subset\mathbb R^n\). The Guillemin--Abreu formalism identifies smooth \(T^n\)-invariant K\"ahler metrics in \(\alpha\) with strictly
convex symplectic potentials on \(\Delta\), modulo affine functions and subject to the standard Guillemin boundary behavior \cite{Guillemin1994,Abreu1998}.

A principal advantage of this description is that scalar curvature
admits the explicit Abreu formula. If \(u\) is a symplectic potential and \((u^{ij})\) is the inverse of its Hessian, then
\[
\Scal(u)
=
-
\sum_{i,j=1}^n
\frac{\partial^2 u^{ij}}
{\partial x_i\partial x_j}.
\]
Thus the search for a torus-invariant metric with prescribed scalar curvature becomes a fourth-order equation for a convex function on the Euclidean polytope \(\Delta\).

Toric manifolds have played an important role in the study of canonical K\"ahler metrics. Building on Guillemin's symplectic description and Abreu's scalar curvature formula, Donaldson formulated the toric cscK problem as a variational problem for convex functions on the moment polytope and introduced the corresponding stability condition \cite{Donaldson2002}. For toric surfaces, Donaldson subsequently proved the existence of cscK metrics under \(K\)-stability, while Chen--Li--Sheng established the corresponding existence result for extremal metrics under relative \(K\)-stability
\cite{Donaldson2009Toric,ChenLiSheng2018}. The relation between relative \(K\)-stability and the modified \(K\)-energy was studied by Zhou--Zhu \cite{ZhouZhu2008}.

For a smooth function \(A\) on \(\overline\Delta\), the prescribed scalar curvature question is translated to the solvability of the Abreu equation
\[
\Scal(u)=A.
\]
Li--Lian--Sheng proved that this equation admits a symplectic potential satisfying the standard Guillemin boundary behavior if and only if \((\Delta,A)\) is relatively \(\widehat K\)-polystable \cite[Theorem~1.3]{LiLianSheng2023}. In Section~\ref{subsec:positive-stable-density}, we prove that every Delzant polytope carries a function \(A\in C^\infty(\overline\Delta)\) such that
\(A(x)>0\) for every \(x\in\overline\Delta \), and such that \((\Delta,A)\) is relatively \(\widehat K\)-polystable. The resulting solution of the Abreu equation
therefore determines a \(T^n\)-invariant K\"ahler metric of positive scalar curvature in the original class \(\alpha\).

Applying Theorem~\ref{thm:variational-main} to this construction gives the following stronger conclusion.

\begin{mainresult}
\label{thm:toric-main}
Let \(X\) be a compact smooth toric K\"ahler manifold with positive dimension. Then every K\"ahler class contains a torus-invariant K\"ahler metric of positive scalar curvature. Consequently, for every \(\alpha\in\mathcal K_X\), every reference K\"ahler metric \(\omega\in\alpha\), and every \(\Omega\in\mathcal V_\alpha^+\), equation \eqref{eq:intro-scalar-measure-equation} admits a smooth solution in \(\Hpot_\omega\), unique up to addition of constants.
\end{mainresult}

\subsection{Criteria for positive scalar curvature K\"ahler metrics}

We define the PSC K\"ahler cone and the total PSC K\"ahler cone by
\[
\mathcal K_X^{\mathrm{psc}}
:=
\left\{
\alpha\in\mathcal K_X:
\alpha\text{ contains a PSC K\"ahler metric}
\right\}
\]
and
\[
\mathcal T_X^+
:=
\left\{
\alpha\in\mathcal K_X:
c_1(X)\cdot \alpha^{n-1}>0
\right\}.
\]
Both are open cones, and the total scalar curvature identity gives the
immediate inclusion
\[
\mathcal K_X^{\mathrm{psc}}
\subseteq
\mathcal T_X^+ \subseteq \mathcal{K}_X.
\]

Theorem~\ref{thm:variational-main} gives an analytic characterization of classes in \(\mathcal K_X^{\mathrm{psc}}\). It is therefore natural to ask whether positive total scalar curvature is sufficient for the existence of a PSC K\"ahler metric in a fixed class.

The answer is negative in general: on non-rational ruled surfaces \(\mathbb{P}(E) \to \Sigma_g\) with \(g \ge2\) and \(E\) unstable, there exist K\"ahler classes with positive total scalar curvature that do not contain any positive scalar curvature K\"ahler metrics (see \cite{Sha26obstruction}). Therefore, it is natural to seek the geometric setting in which the necessary numerical condition is also sufficient.

\begin{question} \label{q:class-level-PSC} 
For which compact K\"ahler manifolds \(X\) does one have 
\(\mathcal K_X^{\mathrm{psc}} = \mathcal T_X^+ \)?
\end{question}

For a Fano manifold, the equality in Question~\ref{q:class-level-PSC} holds by Yau's theorem \cite{Yau1978}. A general geometric characterization of the manifolds satisfying \(\mathcal K_X^{\mathrm{psc}}=\mathcal T_X^+\) remains open.
\begin{maincorollary}
\label{cor:toric-PSC-conjectures}
Let \(X\) be a compact smooth toric K\"ahler manifold of positive complex dimension. Then
\[
\mathcal K_X^{\mathrm{psc}}
=
\mathcal T_X^+
=
\mathcal K_X.
\]
In particular, Questions~\ref{q:class-level-PSC} have an affirmative answer for \(X\).
\end{maincorollary}

\subsection*{On the use of AI}

All mathematical ideas, arguments, and results presented in this manuscript were developed by the author. ChatGPT 5.5 Plus and 5.6 Sol were utilized as an assisting tool to refine the writing and presentation, and to help verify the mathematical logic.

\subsection*{Organization of the paper}

The paper is organized as follows.
Section~\ref{sec:variational} introduces the \(\mathcal F\)-functional and develops its variational theory, including geodesic convexity, uniqueness of the minimizer, \(d_1\)-coercivity, independence of the reference data, and the modified continuity path.
Section~\ref{sec:closedness} establishes the a priori estimates and closedness required for this path.
Section~\ref{sec:geodesic-stability} relates \(d_1\)-coercivity to uniform geodesic stability and completes the proofs of Theorem~\ref{thm:variational-main} and Corollary~\ref{cor:global-parametrization-psc}.
Finally, Section~\ref{sec:toric-PSC} gives the proof of Theorem \ref{thm:toric-main} and Corollary \ref{cor:toric-PSC-conjectures}.

\subsection*{Acknowledgement} The initial idea of this paper was suggested by Prof. Song Sun during the ``International Conference on Complex and Differential Geometry'' at ICTP, Trieste. This suggestion was the starting point of the present work. The author is sincerely grateful to Prof. Xiuxiong Chen for his continued encouragement and support. The author also thanks Longteng Chen, Song Sun, Jian Wang, Mingchen Xia,
and Qi Yao for helpful discussions.

\section{The variational approach}
\label{sec:variational}

Let \(X\) be a compact K\"ahler manifold of complex dimension \(n\), and let \(\alpha\in\mathcal K_X\) be a K\"ahler class. For the scalar curvature, we use the convention
\[
\Scal(\omega)\,\omega^n
=
n\,\Ric(\omega)\wedge\omega^{n-1}.
\]
We write
\[
V_\alpha:=\alpha^n,\qquad
S_\alpha:=2\pi n\,c_1(X)\cdot\alpha^{n-1}
\]
for the volume and total scalar curvature of \(\alpha\), respectively, and set
\[
\overline S_\alpha:=\frac{S_\alpha}{V_\alpha}.
\]
We assume, throughout this paper, that
\(S_\alpha>0\). Fixing a reference
K\"ahler metric \( \omega\in\alpha\), let
\[
        \Hpot_\omega
        =
        \left\{
        \phi\in C^\infty(X,\mathbb R):
        \omega_\phi:=\omega+\ddbar\phi>0
        \right\}
\]
be the space of smooth K\"ahler potentials.
Throughout this section, if
\(\mathcal G:\mathcal H_\omega\to\mathbb R\) is a smooth functional,
we denote its first and second Fréchet differentials at
\(\phi\in\mathcal H_\omega\) by
\[
        D\mathcal G(\phi)[u]
        :=
        \left.
        \frac{d}{dt}
        \right|_{t=0}
        \mathcal G(\phi+t\,u)
\]
and
\[
        D^2\mathcal G(\phi)[u,v]
        :=
        \left.
        \frac{\partial^2}{\partial s\,\partial t}
        \right|_{s=t=0}
        \mathcal G(\phi+s\,u+t\,v),
\]
respectively, where \(u,v\in C^\infty(X,\mathbb R)\).

For a general smooth path \(t\mapsto\phi_t\in\mathcal H_\omega\), we use
\[
        \frac{d}{dt}\mathcal G(\phi_t)
        \qquad\text{and}\qquad
        \frac{d^2}{dt^2}\mathcal G(\phi_t)
\]
for the first and second variations of \(\mathcal G\) along the path. The chain rule gives
\[
        \frac{d^2}{dt^2}\mathcal G(\phi_t)
        =
        D^2\mathcal G(\phi_t)
        [\dot\phi_t,\dot\phi_t]
        +
        D\mathcal G(\phi_t)[\ddot\phi_t].
\]

\subsection{The prescribed scalar curvature measure equation}
\label{subsec:scalar-measure-equation}
As in Section \ref{sec:introduction}, define the space of admissible measures by
\begin{equation*}
       \vap:= \left\{
        \Omega\in
        C^\infty
        \bigl(
        X,\Lambda_{\mathbb R}^{n,n}
        \bigr):
        \Omega>0,\quad
        \int_X\Omega=S_\alpha
        \right\}.
\end{equation*}
For \(\Omega\in\mathcal V_\alpha^+\), we study the prescribed scalar curvature measure equation
\begin{equation}
        \Scal(\omega_\phi)\,\omega_\phi^n=\Omega.
        \label{eq:scalar-measure}
\end{equation}
The compatibility condition \(\Omega\in\mathcal V_\alpha^+\) is necessary, since
\[
        \int_X\Scal(\omega_\phi)\,\omega_\phi^n=S_\alpha.
\]

Let \(E_\omega\) denote the Aubin--Yau energy, also known as the Aubin--Mabuchi energy, normalized by \(E_\omega(0)=0\) and characterized by
\begin{equation}\label{eq:AY-variation}
        DE_\omega(\phi)[u]
        =
        \int_X u\,\omega_\phi^n.
\end{equation}
Equivalently,
\begin{equation}
        E_\omega(\phi)
        =
        \frac{1}{n+1}
        \sum_{j=0}^n
        \int_X
        \phi\,
        \omega_\phi^j\wedge\omega^{n-j}.
        \label{eq:AY}
\end{equation}

Let \(\mathcal M_\omega\) denote the Mabuchi \(K\)-energy. It is characterized up to an additive constant by the differential formula
\begin{equation*}
        D\mathcal M_\omega(\phi)[u]
        =
        -\int_X
        u\,
        \bigl(
        \Scal(\omega_\phi)-\overline S_\alpha
        \bigr)
        \,\omega_\phi^n.
\end{equation*}
We fix its additive normalization by requiring \(\mathcal M_\omega(0)=0\).

The prescribed measure \(\Omega\) determines the linear functional
\begin{equation*}
        L_{\omega,\Omega}(\phi)
        :=
        \int_X \phi\,\Omega.
\end{equation*}
We call \(L_{\omega,\Omega}\) the target linear functional.

Recall that the top-degree \(\mathcal J\)-functional of Chen \cite[Section~6]{Chen2018NewPerspective} associated with \(\Omega\), normalized to vanish at the origin, is
\begin{equation}
        \mathcal J_{\omega,\Omega}(\phi)
        =
        -\overline S_\alpha E_\omega(\phi)
        +
        L_{\omega,\Omega}(\phi).
        \label{eq:top-degree-J}
\end{equation}

We now introduce the prescribed scalar curvature measure functional:
\begin{equation}
        \mathcal F_{\omega,\Omega}(\phi)
        :=
        \mathcal M_\omega(\phi)
        -
        \overline S_\alpha E_\omega(\phi)
        +
        L_{\omega,\Omega}(\phi).
        \label{eq:prescribed-functional}
\end{equation}
Equivalently,
\[
        \mathcal F_{\omega,\Omega}
        =
        \mathcal M_\omega
        +
        \mathcal J_{\omega,\Omega}.
\]

We first see that \(\mathcal F_{\omega,\Omega}\) is invariant under addition of constants.

\begin{lemma}
\label{lem:constant-invariance}
For every \(\phi\in\Hpot_\omega\) and \(c\in\mathbb R\), we have \(\mathcal F_{\omega,\Omega}(\phi+c)=
\mathcal F_{\omega,\Omega}(\phi)\).
\end{lemma}
\begin{proof}
By \eqref{eq:AY}, \(E_\omega(\phi+c) = E_\omega(\phi)+cV_\alpha\). On the other hand,
\(L_{\omega,\Omega}(\phi+c) =L_{\omega,\Omega}(\phi)+cS_\alpha\). Since \(S_\alpha=\overline S_\alpha V_\alpha\), the two constant terms cancel in \eqref{eq:prescribed-functional}. Then the conclusion follows from the fact that the Mabuchi \(K\)-energy is invariant under addition of constants.
\end{proof}

We then compute the Euler--Lagrange equation of \(\mathcal F_{\omega,\Omega}\).
\begin{proposition}
\label{prop:EL}
Let \(\Omega \in \vap\). A potential \(\phi\in\Hpot_\omega\) is a critical point of \(\mathcal F_{\omega,\Omega}\) if and only if \(\phi\) solves \eqref{eq:scalar-measure}. In particular, every critical point determines a PSC K\"ahler metric.
\end{proposition}

\begin{proof}
For every \(u\in C^\infty(X,\mathbb R)\), the differential of \(\mathcal F_{\omega,\Omega}\) at \(\phi\) is
\begin{align}
        D\mathcal F_{\omega,\Omega}(\phi)[u]
        &=
        -\int_X
        u\,
        \bigl(
        \Scal(\omega_\phi)-\overline S_\alpha
        \bigr)
        \,\omega_\phi^n                                      
        -
        \overline S_\alpha
        \int_X
        u\,\omega_\phi^n
        +
        \int_X
        u\,\Omega                                             \\
        &=
        -\int_X
        u\,
        \bigl(
        \Scal(\omega_\phi)\,\omega_\phi^n-\Omega
        \bigr).
        \label{eq:F-differential}
\end{align}
Since \(u\) is arbitrary, the differential vanishes for every
\(u\in C^\infty(X,\mathbb R)\) if and only if
\[
        \Scal(\omega_\phi)\,\omega_\phi^n-\Omega
        =
        0.
\]
This is precisely \eqref{eq:scalar-measure}. Since both \(\Omega\) and \(\omega_\phi^n\) are smooth positive volume forms, the equation implies
\[
        \Scal(\omega_\phi)
        =
        \frac{\Omega}{\omega_\phi^n}
        >
        0,
\]
which completes the proof.
\end{proof}

The use of a fixed measure in \eqref{eq:scalar-measure} is not merely a reformulation of the classical prescribed scalar curvature problem in Riemannian geometry. 
More precisely, fix a smooth reference volume form \(dV_0\) and write \(\Omega=f\,dV_0\). The direct analogue of the usual prescribed scalar curvature problem in a fixed K\"ahler class asks for 
\begin{equation}\label{eq:usual-pre-sc}
    \Scal(\omega_\phi)=f
\end{equation}
as considered in \cite{ChenCheng2021I}, whereas the equation considered here is 
\[ 
\Scal(\omega_\phi)\,\omega_\phi^n=f\,dV_0. 
\] 
Integrating \eqref{eq:usual-pre-sc} with respect to \(\omega_\phi^n\) gives 
\[ 
\int_X f\,\omega_\phi^n=S_\alpha, 
\] 
which still involves the unknown function \(\phi\). 
For a fixed function \(h\), consider the one-form on \(\Hpot_\omega\) defined by 
\[ 
(\beta_h)_\phi(u) := \int_X u\,h\,\omega_\phi^n. 
\]
Its exterior derivative is 
\[ 
(d\beta_h)_\phi(u,v) = n\int_X h\, \bigl( v\,\ddbar u - u\,\ddbar v \bigr) \wedge\omega_\phi^{n-1}. 
\] 
If \(\beta_h\) is closed, then taking \(v=1\) and integrating by parts gives 
\[ 
0 = n\int_X u\,\ddbar h\wedge\omega_\phi^{n-1} 
\] 
for every \(u\in C^\infty(X,\mathbb R)\). 
It follows that 
\[ 
n\ddbar h\wedge\omega_\phi^{n-1}=(\Delta_{\omega_\phi}h)\,\omega_\phi^{n-1} = 0. 
\] 
Since \(X\) is compact, \(h\) must be constant. 

For a fixed smooth measure \(\Omega\), however, the one-form 
\[ 
u \longmapsto \int_X u\,\Omega 
\] 
is independent of the base point \(\phi\) and is exact, 
with primitive 
\[ 
L_{\omega,\Omega}(\phi) = \int_X \phi\,\Omega. 
\] 
Thus prescribing the scalar curvature measure is the natural direct variational prescription 
compatible with the Mabuchi \(K\)-energy in a fixed K\"ahler class.

We emphasize that this does not rule out other auxiliary variational formulations of \eqref{eq:usual-pre-sc}. In particular, it only shows that the direct one-form obtained from the Mabuchi \(K\)-energy by replacing \(\overline S_\alpha\) with a nonconstant function is not closed, and hence does not define a functional on \(\Hpot_\omega\).


\subsection{Relation to higher-degree twisted \(K\)-energy}
\label{subsec:higher-degree-twisted-K-energy}

We now relate the functional \(\mathcal F_{\omega,\Omega}\) to the higher-degree twisted \(K\)-energies introduced by Chen \cite[Section~6]{Chen2018NewPerspective}. Let \(1\leq k\leq n\), and let \(\mu_k\) be a smooth closed real \((k,k)\)-form.  The corresponding higher-degree \(J\)-functional is characterized, up to an additive constant, by
\begin{equation*}
DJ_{\mu_k}(\phi)[u]
=
\int_X
u\,
\left(
\mu_k\wedge\omega_\phi^{n-k}
-
c_k\,\omega_\phi^n
\right),
\qquad
c_k
=
\frac{
[\mu_k]\cdot\alpha^{n-k}
}{
V_\alpha
}.
\label{eq:higher-degree-J-differential}
\end{equation*}
The closedness of \(\mu_k\) ensures that this one-form is closed.

In the top-degree case \(k=n\), take \(\mu_n=\Omega\in\mathcal V_\alpha^+\).  Then
\[
c_n
=
\frac{\int_X\Omega}{V_\alpha}
=
\overline S_\alpha.
\]
After normalizing both sides to vanish at the origin, the corresponding top-degree \(J\)-functional is
\[
J_{\mu_n}
=
\mathcal J_{\omega,\Omega}
=
-\overline S_\alpha E_\omega
+
L_{\omega,\Omega}.
\]
Hence
\[
\mathcal F_{\omega,\Omega}
=
\mathcal M_\omega+\mathcal J_{\omega,\Omega}
=
2\left(
\frac12\,\mathcal M_\omega
+
\frac12\,J_{\mu_n}
\right).
\]
Thus \(\mathcal F_{\omega,\Omega}\) is, up to multiplication by a positive constant, the top-degree twisted \(K\)-energy at the midpoint of the path \(t\mathcal M_\omega+(1-t)J_{\mu_n}\).

The Euler--Lagrange equation for
\[
t\,\mathcal M_\omega+(1-t)J_{\mu_k},
\qquad 0\leq t\leq1,
\]
is
\[
t\,
\bigl(
\Scal(\omega_\phi)-\overline S_\alpha
\bigr)
=
(1-t)
\left(
\frac{
\mu_k\wedge\omega_\phi^{n-k}
}{
\omega_\phi^n
}
-
c_k
\right).
\]
Setting \(k=n\) and \(\mu_n=\Omega\), we obtain
\begin{equation}
t\,
\bigl(
\Scal(\omega_\phi)-\overline S_\alpha
\bigr)
=
(1-t)
\left(
\frac{\Omega}{\omega_\phi^n}
-
\overline S_\alpha
\right).
\label{eq:top-degree-twisted-path}
\end{equation}
At \(t=\frac12\), this is precisely the prescribed scalar curvature measure equation \eqref{eq:scalar-measure}.

Equivalently, setting \(s=\frac{t}{1-t}\) for \(0\leq t<1\), we may rewrite \eqref{eq:top-degree-twisted-path} as
\begin{equation}
\left(
s\,\Scal(\omega_\phi)
+
(1-s)\,\overline S_\alpha
\right)
\omega_\phi^n
=
\Omega.
\label{eq:auxiliary-scaling-equation-2}
\end{equation}
The prescribed scalar curvature measure equation corresponds to \(s=1\).

For every \(s>0\), the linearization of
\eqref{eq:auxiliary-scaling-equation-2} in the potential variable has the same fourth-order elliptic principal part as the scalar curvature operator, multiplied by \(s\).  At \(s=0\), this fourth-order part vanishes and the equation reduces to the complex Monge--Amp\`ere equation
\[
\overline S_\alpha\,\omega_\phi^n=\Omega.
\]
Although this initial equation is solvable by Yau's theorem, the usual implicit-function theorem does not apply at this endpoint.

This loss of order is analogous to the endpoint issue in Chen's \(k=1\) twisted path. In that setting, endpoint openness was proved independently by Zeng \cite{Zeng2019} and Hashimoto \cite{Hashimoto2019}. These results do not directly apply to the path \eqref{eq:auxiliary-scaling-equation-2}.

In Section~\ref{subsec:continuity-path}, we introduce a
modified path for which the coefficient of the Mabuchi \(K\)-energy remains uniformly positive, while the prescribed top-degree form is allowed to vary. Along that path the linearization remains fourth-order elliptic, including at the initial parameter.


\subsection{Geodesic convexity and uniqueness}
\label{subsec:geodesic-convexity-uniqueness}
Recall that a smooth geodesic \(\phi_s\in \Hpot_\omega\) satisfies
\begin{equation*}
        \ddot\phi_s
        -
        |\partial\dot\phi_s|^2_{\omega_{\phi_s}}
        =
        0.
\end{equation*}
It is known that smooth geodesics connecting arbitrary pairs of points of
\(\Hpot_\omega\) are not available in general. Instead, one uses the
weak geodesic segment obtained from the homogeneous complex
Monge--Amp\`ere equation. The existence of weak geodesic segments
joining smooth K\"ahler potentials was established by Chen
\cite{Chen2000Space}, while their full real \(C^{1,1}\)-regularity was
proved by Chu--Tosatti--Weinkove \cite{CTW2017}.

\begin{proposition}
\label{prop:weak-convexity}
The functional \(\mathcal F_{\omega,\Omega}\) is convex along every weak
geodesic segment joining two smooth K\"ahler potentials. Moreover, if
\(s\longmapsto\mathcal F_{\omega,\Omega}(\phi_s)\) is affine along such a weak geodesic segment, then \(\phi_1-\phi_0\) is constant.
\end{proposition}

\begin{proof}
The convexity of \(\mathcal F_{\omega,\Omega}\) follows from
\cite[Theorem~6.2]{Chen2018NewPerspective}. Indeed, by Yau's theorem \cite{Yau1978} there exists \(\eta\in\alpha\) such that \(\Omega=\overline S_\alpha\eta^n\), and, up to multiplication by a positive constant, \(\mathcal F_{\omega,\Omega}\) is the corresponding top-degree twisted \(K\)-energy at \(t=\frac12\).

Suppose now that
\(s\mapsto \mathcal F_{\omega,\Omega}(\phi_s)\)
is affine.  Since both the Mabuchi \(K\)-energy and \(L_{\omega,\Omega}\) are convex, and since \(E_\omega\) is affine, the function \(s\longmapsto L_{\omega,\Omega}(\phi_s)\)
must itself be affine. By the real \(C^{1,1}\)-regularity of the weak geodesic \cite{CTW2017}, \(\ddot\phi_s\ge0\) almost everywhere on \(X\times(0,1)\), and
\[
        0
        =
        \frac{d^2}{ds^2}L_{\omega,\Omega}(\phi_s)
        =
        \int_X\ddot\phi_s\,\Omega
\]
for almost every \(s\in(0,1)\).  Since \(\Omega\) is strictly positive,
we obtain \(\ddot\phi_s=0\) almost everywhere on \(X\times(0,1)\).  Hence
\begin{equation} \label{eq:weak-geodesic-affine}
            \phi_s=(1-s)\phi_0+s\phi_1.
\end{equation}

We next use the homogeneous complex Monge-Amp\`ere formulation of weak geodesics \cite{Chen2000Space}. Let
\[
        \Sigma
        =
        \{\zeta=s+\sqrt{-1}\,r\in\mathbb C:0<s<1,\ r\in\mathbb R\}.
\]
The weak geodesic \(s\mapsto\phi_s\) is represented by the
\(\mathbb R\)-invariant function
\[
        \Phi(x,\zeta)
        =
        \phi_{\operatorname{Re}\zeta}(x)
        =
        \phi_s(x)
\]
on \(X\times\Sigma\). Note that \(\Phi\) is \(\pi_X^*\omega\)-plurisubharmonic and solves
\[
        \bigl(
        \pi_X^*\omega
        +
        \ddbar_{X\times\Sigma}\Phi
        \bigr)^{n+1}
        =
        0,
\]
with boundary values \(\phi_0\) and \(\phi_1\) at \(\operatorname{Re}\zeta=0\) and \(\operatorname{Re}\zeta=1\), respectively.  In particular,
\[
        T
        :=
        \pi_X^*\omega
        +
        \ddbar_{X\times\Sigma}\Phi
\]
is a positive \((1,1)\)-current. By the real \(C^{1,1}\)-regularity of Chu--Tosatti--Weinkove
\cite{CTW2017}, the coefficients of \(T\) are defined almost everywhere. In local product coordinates \((z_1,\ldots,z_n,\zeta)\) on \(X\times\Sigma\), with \(\zeta=s+\sqrt{-1}\,r\), the \(\mathbb R\)-invariance of \(\Phi\) gives
\[
        \Phi_{i\bar\zeta}
        =
        \frac12\,\partial_i\dot\phi_s,
        \qquad
        \Phi_{\zeta\bar\zeta}
        =
        \frac14\,\ddot\phi_s
\]
almost everywhere. Hence the positivity of \(T\) gives a positive semidefinite block matrix of the form
\[
        \begin{pmatrix}
        (\omega_{\phi_s})_{i\bar j}
        &
        \frac12\,\partial_i\dot\phi_s
        \\[4pt]
        \frac12\,\partial_{\bar j}\dot\phi_s
        &
        \frac14\,\ddot\phi_s
        \end{pmatrix}
\]
almost everywhere. Since \(\ddot\phi_s=0\) almost everywhere, the mixed entries of this positive semidefinite matrix vanish. Thus \(\partial\dot\phi_s=0\) almost everywhere. Using \eqref{eq:weak-geodesic-affine}, we obtain \(\partial(\phi_1-\phi_0)=0\). Consequently, \(\phi_1-\phi_0\) is constant.
\end{proof}

The positivity of \(\Omega\) is essential in the preceding argument.
Thus the convexity and uniqueness arguments below are specific to positive prescribed measures and do not formally extend to negative or sign-changing scalar curvature measures.

\begin{lemma}
\label{lem:critical-point-minimizer}
Let \(\phi_0\in\Hpot_\omega\) be a critical point of
\(\mathcal F_{\omega,\Omega}\). Then for every \(\psi\in\Hpot_\omega\),
\[
        \mathcal F_{\omega,\Omega}(\phi_0)
        \leq
        \mathcal F_{\omega,\Omega}(\psi).
\]
\end{lemma}

\begin{proof}
Although the functional is convex along weak geodesics, the initial velocity of such a geodesic need not be smooth. The point is therefore to show that a smooth critical point prevents a downward jump of the right derivative at the initial endpoint. Fix \(\psi\in\Hpot_\omega\), and let
\([0,1]\ni s
        \mapsto
        \phi_s\)
be the weak geodesic segment joining \(\phi_0\) to \(\psi\). Write
\[
        \dot\phi_0^+
        :=
        \lim_{s\downarrow0}
        \frac{\phi_s-\phi_0}{s}.
\]
By the \(C^{1,1}\)-regularity of the weak geodesic, this limit exists
almost everywhere on \(X\) and defines a bounded function.

The endpoint slope inequality for the Mabuchi \(K\)-energy
\cite[Lemma~3.5]{BermanBerndtsson2017} gives
\begin{equation*}
        \lim_{s\downarrow0}
        \frac{
        \mathcal M_\omega(\phi_s)
        -
        \mathcal M_\omega(\phi_0)
        }{s}
        \geq
        -\int_X
        \dot\phi_0^+\,
        \bigl(
        \Scal(\omega_{\phi_0})
        -
        \overline S_\alpha
        \bigr)
        \,\omega_{\phi_0}^n.
\end{equation*}
Since
\( s\mapsto \mathcal M_\omega(\phi_s)\)
is convex, its secant slopes from the initial endpoint are
nondecreasing. Hence, for every \(s\in(0,1]\),
\begin{equation}
\begin{split}
        \frac{
        \mathcal M_\omega(\phi_s)
        -
        \mathcal M_\omega(\phi_0)
        }{s}
        \geq
        -\int_X
        \dot\phi_0^+\,
        \bigl(
        \Scal(\omega_{\phi_0})
        -
        \overline S_\alpha
        \bigr)
        \,\omega_{\phi_0}^n.
        \label{eq:Mabuchi-secant-slope}
\end{split}
\end{equation}
The Aubin--Yau energy is affine along the weak geodesic, and its
initial endpoint derivative is
\[
        \left.
        \frac{d}{ds}
        \right|_{s=0^+}
        E_\omega(\phi_s)
        =
        \int_X
        \dot\phi_0^+\,\omega_{\phi_0}^n.
\]
Consequently, for every \(s\in(0,1]\),
\begin{equation}
        \frac{
        E_\omega(\phi_s)
        -
        E_\omega(\phi_0)
        }{s}
        =
        \int_X
        \dot\phi_0^+\,\omega_{\phi_0}^n.
        \label{eq:E-endpoint-slope}
\end{equation}
Moreover, the pointwise convexity of \(s\mapsto\phi_s(x)\) gives
\[
        \frac{\phi_s-\phi_0}{s}
        \geq
        \dot\phi_0^+
\]
almost everywhere on \(X\). Since \(\Omega\) is positive, integration
yields
\begin{equation}
        \frac{
        L_{\omega,\Omega}(\phi_s)
        -
        L_{\omega,\Omega}(\phi_0)
        }{s}
        \geq
        \int_X
        \dot\phi_0^+\,\Omega.
        \label{eq:L-endpoint-slope}
\end{equation}
By Proposition~\ref{prop:EL}, \(\phi_0\) solves
\[
        \Omega
        =
        \Scal(\omega_{\phi_0})\,\omega_{\phi_0}^n.
\]
Combining
\eqref{eq:Mabuchi-secant-slope},
\eqref{eq:E-endpoint-slope}, and
\eqref{eq:L-endpoint-slope}, we obtain
\begin{align*}
        \frac{
        \mathcal F_{\omega,\Omega}(\phi_s)
        -
        \mathcal F_{\omega,\Omega}(\phi_0)
        }{s}
        \geq 0.
\end{align*}
Taking \(s=1\) gives
\(\mathcal F_{\omega,\Omega}(\psi)
        \geq
        \mathcal F_{\omega,\Omega}(\phi_0)\).
\end{proof}

\begin{corollary}
\label{cor:uniqueness}
For any \(\Omega \in \vap\), there exists at most one
K\"ahler metric in the class \(\alpha\) satisfying
\eqref{eq:scalar-measure}.
\end{corollary}

\begin{proof}
Let \(\phi_0,\phi_1\in\Hpot_\omega\) be two solutions. By
Proposition~\ref{prop:EL}, both are critical points of
\(\mathcal F_{\omega,\Omega}\). Lemma
\ref{lem:critical-point-minimizer} therefore shows that both are global
minimizers. In particular,
\[
        \mathcal F_{\omega,\Omega}(\phi_0)
        =
        \mathcal F_{\omega,\Omega}(\phi_1)
        =
        \inf_{\Hpot_\omega}
        \mathcal F_{\omega,\Omega}.
\]
Let \([0,1]\ni s \mapsto \phi_s\)
be the weak geodesic segment joining \(\phi_0\) to \(\phi_1\).
By convexity,
\[
     \inf_{\Hpot_\omega}
        \mathcal F_{\omega,\Omega}\le   \mathcal F_{\omega,\Omega}(\phi_s)
        \leq
        (1-s)\,
        \mathcal F_{\omega,\Omega}(\phi_0)
        +
        s\,
        \mathcal F_{\omega,\Omega}(\phi_1)=\inf_{\Hpot_\omega}
        \mathcal F_{\omega,\Omega}.
\]
Hence \(s \mapsto \mathcal F_{\omega,\Omega}(\phi_s)\)
is constant. Proposition~\ref{prop:weak-convexity} then implies that \(\phi_1-\phi_0\) is constant.
\end{proof}


\subsection{Linearization operator and its kernel}
\label{subsec:linearization-nondegeneracy}

Define the scalar curvature measure map by
\[
        \mathscr S_\omega:
        \Hpot_\omega
        \longrightarrow
        C^\infty\bigl(X,\Lambda_{\mathbb R}^{n,n}\bigr),
        \qquad
        \mathscr S_\omega(\phi)
        :=
        \Scal(\omega_\phi)\,\omega_\phi^n .
\]
For \(\phi\in\Hpot_\omega\), denote the linearization of
\(\mathscr S_\omega\) at \(\phi\) by
\[
        \mathcal L_{\omega_\phi}u
        :=
        \left.
        \frac{d}{dt}
        \right|_{t=0}
        \mathscr S_\omega(\phi+t\,u).
\]
Since
\[
        \left.
        \frac{d}{dt}
        \right|_{t=0}
        \Ric(\omega_{\phi+t\,u})
        =
        -\ddbar\bigl(\Delta_{\omega_\phi}u\bigr),
\]
we obtain
\begin{equation*}
        \mathcal L_{\omega_\phi}u
        =
        -n\,
        \ddbar\bigl(\Delta_{\omega_\phi}u\bigr)
        \wedge\omega_\phi^{n-1}
        +
        n(n-1)\,
        \Ric(\omega_\phi)\wedge
        \ddbar u\wedge\omega_\phi^{n-2}.
\end{equation*}

We define the scalar operator \(P_{\omega_\phi}\) by
\[
        \mathcal L_{\omega_\phi}u
        =
        P_{\omega_\phi}u\,\omega_\phi^n .
\]
Then
\begin{equation}
        P_{\omega_\phi}u
        =
        -\Delta_{\omega_\phi}^2u
        -
        \left\langle
        \Ric(\omega_\phi),
        \ddbar u
        \right\rangle_{\omega_\phi}
        +
        \Scal(\omega_\phi)\,
        \Delta_{\omega_\phi}u.
        \label{eq:P-operator}
\end{equation}
In particular, \(P_{\omega_\phi}\) is a fourth-order elliptic operator with principal part
\(-\Delta_{\omega_\phi}^2\).

Since
\[
        \int_X
        \mathscr S_\omega(\phi)
        =
        S_\alpha
\]
is independent of \(\phi\), differentiation gives
\[
        \int_X
        P_{\omega_\phi}u\,\omega_\phi^n
        =
        0
\]
for every \(u\in C^\infty(X,\mathbb R)\). For \(k\ge 0\) and \(0<\gamma<1\), set
\[
        C^{k,\gamma}_0(X,\omega_\phi)
        :=
        \left\{
        u\in C^{k,\gamma}(X,\mathbb R):
        \int_X
        u\,\omega_\phi^n
        =
        0
        \right\}.
\]

\begin{proposition}
\label{prop:automatic-nondegeneracy}
Suppose that \(\phi\in\Hpot_\omega\) solves \eqref{eq:scalar-measure} for a fixed \(\Omega \in \vap\). Then
\[
        \ker P_{\omega_\phi}
        =
        \mathbb R.
\]
Moreover, the restriction
\[
        P_{\omega_\phi}:
        C^{4,\gamma}_0(X,\omega_\phi)
        \longrightarrow
        C^{0,\gamma}_0(X,\omega_\phi)
\]
is an isomorphism.
\end{proposition}

\begin{proof}
Differentiating \eqref{eq:F-differential} at \(\phi\) in the direction \(u\),  while keeping \(v\) fixed, and using the definition of \(\mathcal L_{\omega_\phi}\), 
we obtain
\begin{equation}
        D^2\mathcal F_{\omega,\Omega}(\phi)[u,v]
        =
        -\int_X
        v\,P_{\omega_\phi}u\,\omega_\phi^n.
        \label{eq:second-differential-P}
\end{equation}
Since the second differential is symmetric, \eqref{eq:second-differential-P} implies
\[
        \int_X
        v\,P_{\omega_\phi}u\,\omega_\phi^n
        =
        \int_X
        u\,P_{\omega_\phi}v\,\omega_\phi^n.
\]
Thus \(P_{\omega_\phi}\) is formally self-adjoint with respect to \(\omega_\phi^n\).

We next compute the quadratic form at the solution \(\phi\). Along the affine path \(t\mapsto\phi+t\,u\), the second variational formulas give
\[
        D^2\mathcal M_\omega(\phi)[u,u]
        =
        \int_X
        \left|
        \bar\partial
        \nabla_{\omega_\phi}^{1,0}u
        \right|_{\omega_\phi}^2
        \,\omega_\phi^n
        +
        \int_X
        \bigl(
        \Scal(\omega_\phi)-\overline S_\alpha
        \bigr)
        \,
        \lvert\partial u\rvert_{\omega_\phi}^2
        \,\omega_\phi^n
\]
and
\[
        D^2E_\omega(\phi)[u,u]
        =
        -\int_X
        \lvert\partial u\rvert_{\omega_\phi}^2
        \,\omega_\phi^n.
\]
Since \(L_{\omega,\Omega}\) is linear,
\[
        D^2L_{\omega,\Omega}(\phi)[u,u]
        =
        0.
\]
Using \eqref{eq:scalar-measure}, we therefore obtain
\begin{equation}
        -\int_X
        u\,P_{\omega_\phi}u\,\omega_\phi^n
        =
        D^2\mathcal F_{\omega,\Omega}(\phi)[u,u] 
        =
        \int_X
        \left|
        \bar\partial
        \nabla_{\omega_\phi}^{1,0}u
        \right|_{\omega_\phi}^2
        \,\omega_\phi^n                                
        +
        \int_X
        \lvert\partial u\rvert_{\omega_\phi}^2
        \,\Omega.
        \label{eq:P-positive-quadratic-form}
\end{equation}
Let \(u\in C^{4,\gamma}(X,\mathbb R)\) satisfy
\(P_{\omega_\phi}u=0\). Equation \eqref{eq:P-positive-quadratic-form} then gives
\[
        \int_X
        \lvert\partial u\rvert_{\omega_\phi}^2
        \,\Omega
        =
        0.
\]
Since \(\Omega\) is a smooth positive volume form, we obtain \(\partial u=0\). Hence \(u\) is constant, which yields
\[
        \ker P_{\omega_\phi}
        =
        \mathbb R.
\]

Finally, \(P_{\omega_\phi}\) is a formally self-adjoint fourth-order elliptic operator, and therefore has Fredholm index zero. Its image is contained in
\(C^{0,\gamma}_0(X,\omega_\phi)\), since
\[
        \int_X
        P_{\omega_\phi}u\,\omega_\phi^n
        =
        0.
\]
By the Fredholm alternative,
\[
        \operatorname{Im}P_{\omega_\phi}
        =
        C^{0,\gamma}_0(X,\omega_\phi).
\]
Restricting the domain to \(C^{4,\gamma}_0(X,\omega_\phi)\) removes the constant kernel and gives the desired isomorphism.
\end{proof}


\subsection{Finite-energy space and \(d_1\)-coercivity}
\label{subsec:finite-energy-stability}

We briefly recall the finite-energy space used in the variational approach to canonical K\"ahler metrics, and extend \(\mathcal F_{\omega,\Omega}\) to this setting.

Let \(\PSH(X,\omega)\) denote the space of \(\omega\)-plurisubharmonic (psh) functions. The Aubin--Yau energy extends from smooth potentials to \(\PSH(X,\omega)\) by monotone approximation. The finite-energy space is
\[
        \mathcal E^1(X,\omega)
        =
        \left\{
        \phi\in\PSH(X,\omega):
        \int_X\omega_\phi^n=V_\alpha,\quad
        E_\omega(\phi)>-\infty
        \right\},
\]
where \(\omega_\phi^n\) is the non-pluripolar Monge--Amp\`ere measure of \(\phi\).

On \(\Hpot_\omega\), consider the \(L^1\)-type Finsler norm
\begin{equation*}
        \|\xi\|_{1,\phi}
        =
        \frac{1}{V_\alpha}
        \int_X|\xi|\,\omega_\phi^n,
        \qquad
        \xi\in T_\phi\Hpot_\omega.
\end{equation*}
We denote the associated path-length metric by \(d_1\), and we use the same notation for its extension to \(\mathcal E^1(X,\omega)\). In \cite{Darvas2015}, Darvas proved that the metric completion of \((\Hpot_\omega,d_1)\) is naturally identified with \((\mathcal E^1(X,\omega),d_1)\), and that the latter is a complete geodesic metric space.

The Aubin--Yau energy \(E_\omega\) is \(d_1\)-continuous on \(\mathcal E^1(X,\omega)\). The Mabuchi \(K\)-energy admits its canonical lower semicontinuous extension
\[
        \mathcal M_\omega:
        \mathcal E^1(X,\omega)
        \longrightarrow
        \mathbb R\cup\{+\infty\},
\]
which is convex along finite-energy geodesics \cite[Theorem~4.7]{BDL2017}. We also have:
\begin{lemma}
\label{lem:smooth-linear-term-d1-continuity}
Let \(\nu\) be a smooth real top-degree form.  Then the functional
\[
        \phi\longmapsto L_\nu(\phi):=\int_X\phi\,\nu
\]
is finite and \(d_1\)-continuous on \(\mathcal E^1(X,\omega)\). Moreover, if \(\nu>0\), then \(L_\nu\) is convex along finite-energy \(d_1\)-geodesic segments.
\end{lemma}
\begin{proof}
Since \(\nu\) is smooth, there exists \(A>0\), depending only on \(\omega\) and \(\nu\), such that \(|\nu|\leq A\,\omega^n\). Let \(\phi_j\to\phi\) in \(d_1\).  By
\cite[Theorem~5.8]{Darvas2015}, applied with
\(\chi(l)=|l|\) and \(v=0\), we have
\[
        \int_X|\phi_j-\phi|\,\omega^n\longrightarrow0.
\]
Therefore
\[
        \left|
        \int_X(\phi_j-\phi)\,\nu
        \right|
        \leq
        A\int_X|\phi_j-\phi|\,\omega^n
        \longrightarrow0.
\]
This proves the \(d_1\)-continuity.  Finiteness follows from \(\mathcal E^1(X,\omega)\subset L^1(X,\omega^n)\).

Assume now that \(\nu>0\), and let \([0,1]\ni t\mapsto\phi_t\) be a finite-energy \(d_1\)-geodesic segment. As in the proof of Proposition~\ref{prop:weak-convexity}, this geodesic is represented by an \(\mathbb R\)-invariant \(\pi_X^*\omega\)-plurisubharmonic function
\(\Phi(x,\zeta)\) with \(\zeta=t+\sqrt{-1}\,r\in\Sigma\),
on \(X\times\Sigma\). For \(\nu\)-almost every \(x\in X\), the function \(\zeta\longmapsto\Phi(x,\zeta)\)
is subharmonic on \(\Sigma\).  Since \(\Phi\) is \(\mathbb R\)-invariant, this subharmonic function depends only on \(t=\operatorname{Re}\zeta\). Therefore
\(t\mapsto\phi_t(x)\) is convex for \(\nu\)-almost every \(x\in X\). It follows that, for every \(t\in[0,1]\),
\[
        \phi_t(x)
        \leq
        (1-t)\phi_0(x)+t\phi_1(x)
\]
for \(\nu\)-almost every \(x\).  Since
\(\phi_t\in\mathcal E^1(X,\omega)\subset L^1(X,\nu)\), all terms are \(\nu\)-integrable. Integrating the preceding inequality against the
positive measure \(\nu\), we obtain
\[
        L_\nu(\phi_t)
        \leq
        (1-t)L_\nu(\phi_0)+tL_\nu(\phi_1).
\]
Thus \(L_\nu\) is convex along finite-energy \(d_1\)-geodesic segments.
\end{proof}

By Lemma~\ref{lem:smooth-linear-term-d1-continuity},
\(L_{\omega,\Omega}\) extends uniquely and \(d_1\)-continuously to \(\mathcal E^1(X,\omega)\). Therefore, we can extend
\[
        \mathcal F_{\omega,\Omega}
        =
        \mathcal M_\omega
        -
        \overline S_\alpha E_\omega
        +
        L_{\omega,\Omega}
\]
to \(\mathcal E^1(X,\omega)\) by the same formula. This extension is \(d_1\)-lower semicontinuous.  It is convex along finite-energy geodesics, because \(\mathcal M_\omega\) is convex, \(E_\omega\) is affine, and \(L_{\omega,\Omega}\) is convex for \(\Omega>0\).

It follows from Lemma \ref{lem:constant-invariance} that \(\mathcal F_{\omega,\Omega}\) is invariant under addition of constants. We therefore regard \(\mathcal F_{\omega,\Omega}\) as living on \(\mathcal E^1(X,\omega)/\mathbb R\), and represent this quotient by the canonical slice
\begin{equation*}
        \mathcal E^1_0(X,\omega)
        =
        \left\{
        \phi\in\mathcal E^1(X,\omega):
        E_\omega(\phi)=0
        \right\}.
\end{equation*}
Indeed,
\[
        E_\omega(\phi+c)=E_\omega(\phi)+cV_\alpha,
\]
so every class modulo constants has a unique representative in \(\mathcal E^1_0(X,\omega)\). By our normalization \(\mathcal{M}_\omega(0)=0\), hence \(\mathcal F_{\omega,\Omega}(0)=0\) on \(\mathcal E^1_0(X,\omega)\). We can now give the formal definition of \(d_1\)-coercivity of the functional \(\mathcal{F}_{\omega,\Omega}\) for a fixed prescribed measure \(\Omega\in\vap\) and a fixed reference metric \(\omega\in\alpha\).

\begin{definition}[\(d_1\)-coercivity]
\label{def:d1-coercivity}
We say that \(\mathcal F_{\omega,\Omega}\) is \emph{\(d_1\)-coercive} if there exist constants \(\delta>0\) and \(C>0\) such that
\begin{equation}
        \mathcal F_{\omega,\Omega}(\phi)
        \ge
        \delta\,d_1(0,\phi)-C
        \label{eq:d1-coercivity}
\end{equation}
for every \(\phi\in\mathcal E^1_0(X,\omega)\).
\end{definition}

This is the analogue of the \(d_1\)-coercivity for the Mabuchi \(K\)-energy  in the variational theory of canonical K\"ahler metrics (see, for example, \cite{DR2017}).

We next introduce the radial slope from an arbitrary base point in the normalized finite-energy space. Let
\(\ell:[0,\infty)\rightarrow\mathcal E^1_0(X,\omega)\)
be a finite-energy \(d_1\)-geodesic ray parameterized with constant \(d_1\)-speed, and assume that
\(\mathcal F_{\omega,\Omega}(\ell_0)<+\infty\).
Set \(v_1(\ell):=d_1(\ell_0,\ell_1)\). Thus
\[
        d_1(\ell_r,\ell_s)
        =
        |r-s|\,v_1(\ell),
        \qquad
        r,s\geq0.
\]

Since \(\mathcal F_{\omega,\Omega}\) is convex along finite-energy geodesics, the quotient
\[
        \frac{
        \mathcal F_{\omega,\Omega}(\ell_t)
        -
        \mathcal F_{\omega,\Omega}(\ell_0)
        }{t}
\]
is nondecreasing for \(t>0\). We define the radial slope of \(\mathcal F_{\omega,\Omega}\) along \(\ell\) by
\[
        \mathcal F_{\omega,\Omega}\{\ell\}
        :=
        \lim_{t\to\infty}
        \frac{
        \mathcal F_{\omega,\Omega}(\ell_t)
        -
        \mathcal F_{\omega,\Omega}(\ell_0)
        }{t}
        \in(-\infty,+\infty].
\]
Because \(\mathcal F_{\omega,\Omega}(\ell_0)<+\infty\), equivalently,
\[
        \mathcal F_{\omega,\Omega}\{\ell\}
        =
        \lim_{t\to\infty}
        \frac{\mathcal F_{\omega,\Omega}(\ell_t)}{t}.
\]
In particular, when \(\ell_0=0\), this agrees with the convention used in Section~\ref{sec:introduction}.

\begin{proposition}
\label{prop:coercivity-implies-radial-bound}
Suppose that \(\mathcal F_{\omega,\Omega}\) is \(d_1\)-coercive with constants \(\delta>0\) and \(C>0\). Then, for every finite-energy \(d_1\)-geodesic ray
\(\ell:[0,\infty)\rightarrow\mathcal E^1_0(X,\omega)\)
parameterized with constant \(d_1\)-speed and satisfying
\(\mathcal F_{\omega,\Omega}(\ell_0)<+\infty\),
one has
\[
        \mathcal F_{\omega,\Omega}\{\ell\}
        \geq
        \delta\,v_1(\ell).
\]
\end{proposition}
\begin{proof}
It follows from the triangle inequality,
\[
\begin{aligned}
        d_1(0,\ell_t)
        &\geq
        d_1(\ell_0,\ell_t)-d_1(0,\ell_0)
        \\
        &=
        t\,v_1(\ell)-d_1(0,\ell_0).
\end{aligned}
\]
Then, by the \(d_1\)-coercivity of \(\mathcal F_{\omega,\Omega}\),
\[
\begin{aligned}
        \mathcal F_{\omega,\Omega}(\ell_t)
        -
        \mathcal F_{\omega,\Omega}(\ell_0)
        \geq
        \delta\,t\,v_1(\ell)
        -
        \delta\,d_1(0,\ell_0)
        -
        C
        -
        \mathcal F_{\omega,\Omega}(\ell_0).
\end{aligned}
\]
Dividing by \(t\) and letting \(t\to\infty\) gives
\[
        \mathcal F_{\omega,\Omega}\{\ell\}
        \geq
        \delta\,v_1(\ell),
\]
which completes the proof.
\end{proof}


\subsection{Independence of the reference metric and prescribed measure}
\label{subsec:target-independence}

We first record the transformation rule for
\(\mathcal F_{\omega,\Omega}\) under a change of reference metric.
\begin{proposition}
\label{prop:change-reference-metric}
Let \(\omega,\eta\in\alpha\) be two K\"ahler metrics, so that \(\eta=\omega+\ddbar\psi\) for \(\psi \in C^\infty(X,\mathbb R)\). Then for every \(u \in \mathcal E^1(X,\eta)\) and \(\Omega \in \vap\),
\begin{equation*}
        \mathcal F_{\eta,\Omega}(u)
        =
        \mathcal F_{\omega,\Omega}(u+\psi)
        -
        \mathcal F_{\omega,\Omega}(\psi).
\end{equation*}
\end{proposition}
\begin{proof}
For \(u\in\Hpot_\eta\), we have \(\eta+\ddbar u = \omega+\ddbar(\psi+u)\). Thus \(\eta_u=\omega_{\psi+u}\).
By the cocycle identities, we have
\[
        E_\eta(u)
        =
        E_\omega(\psi+u)-E_\omega(\psi).
\]
Similarly, 
\[
        \mathcal M_\eta(u)
        =
        \mathcal M_\omega(\psi+u)-\mathcal M_\omega(\psi).
\]
Finally,
\[
        L_{\eta,\Omega}(u)
        =
        \int_Xu\,\Omega
        =
        \int_X(\psi+u)\Omega-\int_X\psi\,\Omega
        =
        L_{\omega,\Omega}(\psi+u)-L_{\omega,\Omega}(\psi).
\]
Combining the three identities gives the desired equality on smooth potentials. In particular, the equality extends to \(\mathcal E^1\) by the cocycle property of the extended Mabuchi \(K\)-energy and the \(d_1\)-continuity of \(E_\omega\) and \(L_{\omega,\Omega}\).
\end{proof}

\begin{corollary}\label{cor:reference-independent}
Let \(\omega,\eta\in\alpha\) be two K\"ahler metrics, and let \(\Omega\in\mathcal V_\alpha^+\). Then \(\mathcal F_{\eta,\Omega}\) is \(d_1\)-coercive on
\(\mathcal E^1_0(X,\eta)\) if and only if \(\mathcal F_{\omega,\Omega}\) is \(d_1\)-coercive on \(\mathcal E^1_0(X,\omega)\).
\end{corollary}
\begin{proof}
Write \(\eta=\omega+\ddbar\psi\) for some \(\psi\in C^\infty(X,\mathbb R)\), and set
\[
        c_\psi
        :=
        \frac{E_\omega(\psi)}{V_\alpha},
        \qquad
        \widehat\psi
        :=
        \psi-c_\psi.
\]
By the cocycle identity for the Aubin--Yau energy, the map
\[
        \widehat T_\psi:
        \mathcal E^1_0(X,\eta)
        \longrightarrow
        \mathcal E^1_0(X,\omega),
        \qquad
        \widehat T_\psi(u):=u+\widehat\psi,
\]
is a bijection. The change of reference metric \(u \mapsto u+\psi\) is a \(d_1\)-isometry between the corresponding finite-energy spaces, and hence
\[
        d_{1,\eta}(0,u)
        =
        d_{1,\omega}
        \bigl(
        \widehat\psi,\widehat T_\psi(u)
        \bigr).
\]
Consequently, by the triangle inequality,
\begin{equation}
\left|
        d_{1,\omega}
        \bigl(
        0,\widehat T_\psi(u)
        \bigr)
        -
        d_{1,\eta}(0,u)
\right|
\leq
        d_{1,\omega}(0,\widehat\psi).
\label{eq:reference-change-d1-comparison}
\end{equation}
By Proposition~\ref{prop:change-reference-metric},
\begin{equation}
        \mathcal F_{\eta,\Omega}(u)
        =
        \mathcal F_{\omega,\Omega}
        \bigl(
        \widehat T_\psi(u)
        \bigr)
        -
        \mathcal F_{\omega,\Omega}(\widehat\psi).
\label{eq:reference-change-normalized-F}
\end{equation}

Suppose that \(\mathcal F_{\omega,\Omega}\) is \(d_1\)-coercive. Then, for some \(\delta>0\) and \(C>0\),
\[
        \mathcal F_{\omega,\Omega}(v)
        \geq
        \delta\,d_{1,\omega}(0,v)-C
\]
for every \(v\in\mathcal E^1_0(X,\omega)\). Applying this to \(v=\widehat T_\psi(u)\) and using
\eqref{eq:reference-change-d1-comparison} and
\eqref{eq:reference-change-normalized-F}, we obtain
\[
        \mathcal F_{\eta,\Omega}(u)
        \geq
        \delta\,d_{1,\eta}(0,u)-C',
\]
where \(C'>0\) is independent of \(u\). Thus \(\mathcal F_{\eta,\Omega}\) is \(d_1\)-coercive. The converse follows by interchanging \(\omega\) and \(\eta\).
\end{proof}

The preceding corollary shows that \(d_1\)-coercivity is independent of the reference metric when the prescribed measure is fixed. We now fix a reference metric \(\omega\in\alpha\) and compare the functionals associated with different admissible measures. We show that changing the admissible measure changes \(\mathcal F_{\omega,\Omega}\) only by a uniformly bounded term on \(\mathcal E^1(X,\omega)\).

Let \(\Omega_1,\Omega_2\in \mathcal V_\alpha^+\). By Yau's theorem \cite{Yau1978}, there are unique K\"ahler forms \(\eta_1,\eta_2\in\alpha\) such that
\begin{equation}
        \eta_i^n
        =
        \frac{\Omega_i}{\overline S_\alpha},
        \qquad
        i=1,2.
        \label{eq:eta-target}
\end{equation}

\begin{proposition}
\label{prop:bounded-target-difference}
There exists a constant \(C=C(\alpha,\Omega_1,\Omega_2)>0\) such that
\begin{equation}
        \left|
        \int_X
        \phi\,
        (\Omega_1-\Omega_2)
        \right|
        \leq
        C
        \label{eq:bounded-target-term}
\end{equation}
for every \(\phi\in\mathcal E^1(X,\omega)\).
\end{proposition}

\begin{proof}
Since \(\eta_1,\eta_2\in\alpha\), the \(\partial \bar\partial\)-lemma gives a smooth real-valued function \(u\), unique up to an additive constant, such that \(\eta_1-\eta_2=\ddbar u\). Fix any normalization of \(u\), and set
\[
        T
        :=
        \sum_{j=0}^{n-1}
        \eta_1^j\wedge\eta_2^{n-1-j}.
\]
Then \(T\) is a smooth closed positive
\((n-1,n-1)\)-form and
\[
        \eta_1^n-\eta_2^n
        =
        \ddbar u\wedge T.
\]
We first prove the estimate for \(\phi\in\Hpot_\omega\). By \eqref{eq:eta-target} and integration by parts,
\begin{align*}
        \int_X
        \phi\,
        (\Omega_1-\Omega_2)
        &=
        \overline S_\alpha
        \int_X
        \phi\,
        \ddbar u\wedge T                                  \\
        &=
        \overline S_\alpha
        \int_X
        u\,
        \ddbar\phi\wedge T                                 \\
        &=
        \overline S_\alpha
        \int_X
        u\,
        (\omega_\phi-\omega)\wedge T.
\end{align*}
Since \([T] = n\,\alpha^{n-1}\), we have
\[
        \int_X
        \omega_\phi\wedge T
        =
        \int_X
        \omega\wedge T
        =
        n\,V_\alpha.
\]
The positivity of \(\omega_\phi\), \(\omega\), and \(T\) therefore gives
\begin{align*}
        \left|
        \int_X
        \phi\,
        (\Omega_1-\Omega_2)
        \right|
        &\leq
        \overline S_\alpha
        \lVert u\rVert_{C^0}
        \left(
        \int_X
        \omega_\phi\wedge T
        +
        \int_X
        \omega\wedge T
        \right)                                           \\
        &=
        2n\,\overline S_\alpha V_\alpha
        \lVert u\rVert_{C^0}.
\end{align*}
This proves \eqref{eq:bounded-target-term} on \(\Hpot_\omega\).

The functional
\[
        \phi
        \longmapsto
        \int_X
        \phi\,
        (\Omega_1-\Omega_2)
\]
is \(d_1\)-continuous on \(\mathcal E^1(X,\omega)\), while \(\Hpot_\omega\) is \(d_1\)-dense in \(\mathcal E^1(X,\omega)\). Hence the same estimate extends to \(\mathcal E^1(X,\omega)\).
\end{proof}

\begin{corollary}
\label{cor:target-independent-coercivity}
Let \(\Omega_1,\Omega_2\in\mathcal V_\alpha^+\). Then:

\begin{enumerate}
\item
\(\mathcal F_{\omega,\Omega_1}\) is \(d_1\)-coercive on
\(\mathcal E^1_0(X,\omega)\) if and only if
\(\mathcal F_{\omega,\Omega_2}\) is \(d_1\)-coercive on
\(\mathcal E^1_0(X,\omega)\).

\item
For every finite-energy \(d_1\)-geodesic ray \(\ell:[0,\infty)\rightarrow\mathcal E^1_0(X,\omega)\), provided that \(\mathcal F_{\omega,\Omega_i}(\ell_0)<+\infty\) for \(i\in\{1,2\}\), one has \( \mathcal F_{\omega,\Omega_1}\{\ell\}=\mathcal F_{\omega,\Omega_2}\{\ell\}\).
\end{enumerate}
\end{corollary}

\begin{proof}
For every \(\phi\in\mathcal E^1(X,\omega)\),
\[
        \mathcal F_{\omega,\Omega_1}(\phi)
        -
        \mathcal F_{\omega,\Omega_2}(\phi)
        =
        \int_X\phi\,(\Omega_1-\Omega_2).
\]
By Proposition~\ref{prop:bounded-target-difference}, the difference between the two functionals is uniformly bounded on \(\mathcal E^1(X,\omega)\). The first assertion follows immediately by restricting to \(\mathcal E^1_0(X,\omega)\).

The same boundedness also shows that
\[
        \mathcal F_{\omega,\Omega_1}(\ell_0)<+\infty
        \quad\Longleftrightarrow\quad
        \mathcal F_{\omega,\Omega_2}(\ell_0)<+\infty.
\]
Moreover, for some \(C>0\),
\[
\begin{aligned}
\left|
        \frac{
        \mathcal F_{\omega,\Omega_1}(\ell_t)
        -
        \mathcal F_{\omega,\Omega_1}(\ell_0)
        }{t}
        -
        \frac{
        \mathcal F_{\omega,\Omega_2}(\ell_t)
        -
        \mathcal F_{\omega,\Omega_2}(\ell_0)
        }{t}
\right|
        \leq
        \frac{2C}{t}.
\end{aligned}
\]
Letting \(t\to\infty\) proves the second assertion.
\end{proof}

Combining Corollaries~\ref{cor:reference-independent}
and~\ref{cor:target-independent-coercivity}, we conclude that the \(d_1\)-coercivity of \(\mathcal F_{\omega,\Omega}\) is independent of the choice of reference metric and admissible measure. We therefore say that the \(\mathcal F\)-functional associated with \(\alpha\) is \(d_1\)-coercive if \(\mathcal F_{\omega,\Omega}\) is \(d_1\)-coercive for one,
equivalently every, choice of \(\omega\in\alpha\) and 
\(\Omega\in\mathcal V_\alpha^+\).

We now turn to radial slopes. Two finite-energy \(d_1\)-geodesic rays \(\ell\) and \(\widetilde\ell\) are called \emph{parallel} if
\[
        \sup_{t\geq0}
        d_1(\ell_t,\widetilde\ell_t)
        <
        +\infty.
\]
By \cite[Proposition~4.1]{DL2020}, for every
\(u,v\in\mathcal E^1(X,\omega)\) and every finite-energy
\(d_1\)-geodesic ray starting from \(u\), there exists a unique parallel ray starting from \(v\). Moreover,
\[
        d_1(\ell_t,\widetilde\ell_t)
        \leq
        d_1(u,v),
        \qquad
        t\geq0.
\]
\begin{lemma}
\label{lem:base-point-invariance-radial-slopes}
Let \(\ell,\widetilde\ell:[0,\infty)\longrightarrow\mathcal E^1(X,\omega)\) be parallel finite-energy \(d_1\)-geodesic rays. Assume that
\(\mathcal F_{\omega,\Omega}(\ell_0)<+\infty\) and
\( \mathcal F_{\omega,\Omega}(\widetilde\ell_0)<+\infty\).
Then the two rays have the same constant \(d_1\)-speed and
\[
        \mathcal F_{\omega,\Omega}\{\ell\}
        =
        \mathcal F_{\omega,\Omega}\{\widetilde\ell\}.
\]
\end{lemma}

\begin{proof}
The triangle inequality gives
\[
\left|
        d_1(\ell_0,\ell_t)
        -
        d_1(\widetilde\ell_0,\widetilde\ell_t)
\right|
\leq
        d_1(\ell_0,\widetilde\ell_0)+ \sup_{t\geq0}
        d_1(\ell_t,\widetilde\ell_t).
\]
Since both rays have constant speed, dividing by \(t\) and letting \(t\to\infty\) shows that \( v_1(\ell)=v_1(\widetilde\ell)\).

For the radial slopes, we apply the argument of
\cite[Lemma~4.10]{DL2020}. That argument uses only geodesic convexity and \(d_1\)-lower semicontinuity of the functional, both of which hold for \(\mathcal F_{\omega,\Omega}\).
More precisely, by the reduction used in the proof of
\cite[Proposition~4.1]{DL2020}, it is enough first to consider the case in which the two initial points are ordered. For \(T>s>0\), let \([0,T]\ni r\mapsto\gamma_r^T\) be the finite-energy geodesic segment joining \(\widetilde\ell_0\) to \(\ell_T\). Then
\(\gamma_s^T \rightarrow \widetilde\ell_s\) in \(d_1\) as \(T\to\infty\). Geodesic convexity then gives
\[
        \mathcal F_{\omega,\Omega}(\gamma_s^T)
        \leq
        \left(
        1-\frac{s}{T}
        \right)
        \mathcal F_{\omega,\Omega}(\widetilde\ell_0)
        +
        \frac{s}{T}
        \mathcal F_{\omega,\Omega}(\ell_T).
\]
Using \(d_1\)-lower semicontinuity and letting \(T\to\infty\), we obtain
\[
        \mathcal F_{\omega,\Omega}(\widetilde\ell_s)
        -
        \mathcal F_{\omega,\Omega}(\widetilde\ell_0)
        \leq
        s\,\mathcal F_{\omega,\Omega}\{\ell\}.
\]
Dividing by \(s\) and letting \(s\to\infty\) gives
\[
        \mathcal F_{\omega,\Omega}\{\widetilde\ell\}
        \leq
        \mathcal F_{\omega,\Omega}\{\ell\}.
\]
The general case follows by inserting a parallel ray based at a common smooth upper bound of \(\ell_0\) and \(\widetilde\ell_0\), exactly as in the proof of \cite[Proposition~4.1 and Lemma~4.10]{DL2020}. Interchanging the two rays gives the reverse inequality.
\end{proof}

We now verify that the radial slope is independent of the reference metric and admissible measure. Let
\[\eta=\omega+\ddbar\psi,
        \qquad
        \widehat\psi
        =
        \psi-\frac{E_\omega(\psi)}{V_\alpha},
\]
and let \(\ell^\eta:[0,\infty)\rightarrow \mathcal E^1_0(X,\eta)\)
be a finite-energy \(d_1\)-geodesic ray starting from \(0\). Under the normalized change of reference coordinates, set
\[
        \ell'_t
        :=
        \ell^\eta_t+\widehat\psi.
\]
Then \(\ell'\) is a finite-energy \(d_1\)-geodesic ray in
\(\mathcal E^1_0(X,\omega)\), starting from \(\widehat\psi\), and it has the same constant \(d_1\)-speed as \(\ell^\eta\). By Proposition~\ref{prop:change-reference-metric},
\[
        \mathcal F_{\eta,\Omega}\{\ell^\eta\}
        =
        \mathcal F_{\omega,\Omega}\{\ell'\}.
\]

Let \(\widetilde\ell\) be the unique ray starting from
\(0\in\mathcal E^1(X,\omega)\) and parallel to \(\ell'\). We claim that \(\widetilde\ell\) remains in \(\mathcal E^1_0(X,\omega)\). Indeed, \(E_\omega\) is affine along finite-energy geodesics and satisfies
\[
        |E_\omega(u)-E_\omega(v)|
        \leq
        V_\alpha d_1(u,v).
\]
Since \(E_\omega(\ell'_t)=0\) and the two rays remain at uniformly bounded \(d_1\)-distance, the affine function \(t\mapsto E_\omega(\widetilde\ell_t)\) is bounded. Since it vanishes at \(t=0\), it must vanish identically. Thus
\[
        \widetilde\ell_t\in\mathcal E^1_0(X,\omega),
        \qquad
        t\geq0.
\]

Lemma~\ref{lem:base-point-invariance-radial-slopes} now gives
\[
        \mathcal F_{\omega,\Omega}\{\ell'\}
        =
        \mathcal F_{\omega,\Omega}\{\widetilde\ell\},
        \qquad
        v_1(\ell^\eta)
        =
        v_1(\widetilde\ell).
\]
Together with Corollary~\ref{cor:target-independent-coercivity}, this shows that, after replacing a ray by the corresponding parallel ray based at the origin, its radial slope and constant \(d_1\)-speed are independent of the reference metric, the admissible measure, and the base point. We therefore write
\begin{equation}
        \mathcal F_\alpha\{\ell\}
        :=
        \mathcal F_{\omega,\Omega}\{\ell\}.
        \label{eq:target-independent-radial-functional}
\end{equation}
Here rays based at the origins of different normalized potential spaces are identified by the normalized change of reference coordinates followed by parallelism.

\begin{definition}[Uniform geodesic stability]
\label{def:uniform-geodesic-stability}
We say that the \(\mathcal F\)-functional associated with
\(\alpha\) is \emph{uniformly geodesically stable} if there exists a constant \(\delta>0\) such that for any reference metric \(\omega\in\alpha\),
\[
        \mathcal F_\alpha\{\ell\}
        \geq
        \delta\,v_1(\ell)
\]
for every nonconstant finite-energy \(d_1\)-geodesic ray
\(\ell:[0,\infty)\rightarrow\mathcal E^1_0(X,\omega)\)
starting from \(0\) and parameterized with constant \(d_1\)-speed.
\end{definition}

\begin{corollary}
\label{cor:analytic-PSC-stability-implies-radial-stability}
If the \(\mathcal F\)-functional associated with \(\alpha\) is \(d_1\)-coercive, then it is uniformly geodesically stable.
\end{corollary}
\begin{proof}
Fix any \(\omega\in\alpha\) and \(\Omega\in\vap\).
Since \(\mathcal F_{\omega,\Omega}\) is \(d_1\)-coercive on \(\mathcal E^1_0(X,\omega)\), Proposition~\ref{prop:coercivity-implies-radial-bound} gives a constant \(\delta>0\) such that
\[
        \mathcal F_{\omega,\Omega}\{\ell\}
        \geq
        \delta\,v_1(\ell)
\]
for every nonconstant finite-energy \(d_1\)-geodesic ray in \(\mathcal E^1_0(X,\omega)\) starting from \(0\).
By \eqref{eq:target-independent-radial-functional}, the left-hand side equals \(\mathcal F_\alpha\{\ell\}\), which proves the claim.
\end{proof}


\subsection{The continuity path}
\label{subsec:continuity-path}

In this subsection, we introduce the modified path. Let \(\phi_0\in\Hpot_\omega\) be the unique normalized solution of the Monge--Amp\`ere equation
\begin{equation*}
        \eta^n
        =
        \frac{\Omega}{\overline S_\alpha},
        \qquad
        \eta
        :=
        \omega_{\phi_0},
        \qquad
        E_\omega(\phi_0)
        =
        0.
        \label{eq:Yau-initial-metric}
\end{equation*}
Choose \(0<\varepsilon_0<1\) sufficiently small so that
\begin{equation}
        \varepsilon_0
        \left\lVert
        \Scal(\eta)-\overline S_\alpha
        \right\rVert_{C^0(X)}
        <
        \frac{\overline S_\alpha}{2}.
        \label{eq:epsilon-zero-choice}
\end{equation}
Such a choice is possible because \(\overline S_\alpha>0\). Define 
\[
        q_0
        :=
        \varepsilon_0\Scal(\eta)
        +
        (1-\varepsilon_0)\overline S_\alpha
        =
        \overline S_\alpha
        +
        \varepsilon_0
        \bigl(
        \Scal(\eta)-\overline S_\alpha
        \bigr),
\]
and set
\begin{equation}
        \Omega_0:=q_0\,\eta^n .
        \label{eq:Omega-zero-definition}
\end{equation}
By \eqref{eq:epsilon-zero-choice},
\(q_0 \geq
        \frac12 \overline S_\alpha
        >
        0\),
so \(\Omega_0\) is a smooth positive volume form. Moreover,
\begin{align*}
        \int_X
        \Omega_0
        &=
        \varepsilon_0
        \int_X
        \Scal(\eta)\,\eta^n
        +
        (1-\varepsilon_0)\overline S_\alpha
        \int_X
        \eta^n                                                \\
        &=
        \varepsilon_0S_\alpha
        +
        (1-\varepsilon_0)\overline S_\alpha V_\alpha            \\
        &=
        S_\alpha.
\end{align*}
Hence \(\Omega_0\in\mathcal V_\alpha^+\). For \(t\in[0,1]\), define
\[
        a_t
        :=
        \varepsilon_0+(1-\varepsilon_0)t,
        \qquad
        \Omega_t
        :=
        (1-t)\Omega_0+t\,\Omega .
\]
Then \(\varepsilon_0 \leq a_t\leq1\) and \(\Omega_t \in \mathcal V_\alpha^+\) for every \(t\in[0,1]\).

Define the modified path functional
\begin{equation}
        \mathcal G_t
        :=
        a_t\mathcal M_\omega
        -
        \overline S_\alpha E_\omega
        +
        L_{\omega,\Omega_t}
        \label{eq:modified-path-functional}
\end{equation}
with the first variational formula
\begin{equation}
        D\mathcal G_t(\phi)[u]
        =
        -\int_X
        u\,
        \left[
        \left(
        a_t\Scal(\omega_\phi)
        +
        (1-a_t)\overline S_\alpha
        \right)
        \omega_\phi^n
        -
        \Omega_t
        \right].
        \label{eq:modified-path-differential}
\end{equation}
Consequently, the Euler--Lagrange equation of
\(\mathcal G_t\) is
\begin{equation}
        \left(
        a_t\Scal(\omega_\phi)
        +
        (1-a_t)\overline S_\alpha
        \right)
        \omega_\phi^n
        =
        \Omega_t.
        \label{eq:modified-scalar-measure-path}
\end{equation}
At \(t=0\), one has \(a_0=\varepsilon_0\), and by
\eqref{eq:Omega-zero-definition} the previously chosen potential \(\phi_0\) solves
\[
        \left(
        \varepsilon_0\Scal(\eta)
        +
        (1-\varepsilon_0)\overline S_\alpha
        \right)
        \eta^n
        =
        \Omega_0.
\]
At \(t=1\), one has \(a_1 =1\),m\(\Omega_1=\Omega\), and \eqref{eq:modified-scalar-measure-path} reduces to the
equation
\[
        \Scal(\omega_\phi)\,\omega_\phi^n
        =
        \Omega.
\]

Define \(\mathcal R:[0,1]\times\Hpot_\omega\rightarrow C^\infty\bigl(X,\Lambda_{\mathbb R}^{n,n}\bigr)\)
by
\[
        \mathcal R(t,\phi)
        :=
        \left(
        a_t\Scal(\omega_\phi)
        +
        (1-a_t)\overline S_\alpha
        \right)
        \omega_\phi^n
        -
        \Omega_t.
\]
For fixed \(t\), define the linearized operator
\(P_{t,\phi}\) by
\[
        D_\phi\mathcal R(t,\phi)[u]
        =
        P_{t,\phi}u\,\omega_\phi^n.
\]
Using \eqref{eq:P-operator}, we obtain
\begin{equation*}
        P_{t,\phi}u
        =
        a_tP_{\omega_\phi}u
        +
        (1-a_t)\overline S_\alpha
        \Delta_{\omega_\phi}u.
\end{equation*}
Since \(a_t\geq\varepsilon_0>0\), the operator
\(P_{t,\phi}\) is fourth-order elliptic for every  \(t\in[0,1]\), with principal part \( -a_t\Delta_{\omega_\phi}^2\).

Differentiating \eqref{eq:modified-path-differential} in the potential variable gives
\begin{equation}
        D^2\mathcal G_t(\phi)[u,v]
        =
        -\int_X
        v\,P_{t,\phi}u\,\omega_\phi^n.
        \label{eq:modified-path-second-differential}
\end{equation}
The symmetry of the second differential implies that
\(P_{t,\phi}\) is formally self-adjoint with respect to
\(\omega_\phi^n\). If \(\phi\) solves \eqref{eq:modified-scalar-measure-path}, the same computation as in Proposition~\ref{prop:automatic-nondegeneracy} gives
\begin{equation*}
        -\int_X
        u\,P_{t,\phi}u\,\omega_\phi^n
        =
        a_t
        \int_X
        \left|
        \bar\partial\nabla_{\omega_\phi}^{1,0}u
        \right|_{\omega_\phi}^2
        \,\omega_\phi^n                                      
        +
        \int_X
        \lvert\partial u\rvert_{\omega_\phi}^2
        \,\Omega_t.
        \label{eq:modified-path-quadratic-form}
\end{equation*}
In particular, since \(\Omega_t>0\), \(\ker P_{t,\phi}= \mathbb R\) at every smooth solution.

\begin{proposition}
\label{prop:modified-path-openness}
Let \(t_*\in[0,1]\), and suppose that \(\phi_*\in\Hpot_\omega\) solves \eqref{eq:modified-scalar-measure-path} at \(t=t_*\). Then
\[
        P_{t_*,\phi_*}:
        C^{4,\gamma}_0(X,\omega_{\phi_*})
        \longrightarrow
        C^{0,\gamma}_0(X,\omega_{\phi_*})
\]
is an isomorphism. Consequently, the set
\begin{equation*}
        \mathcal T
        :=
        \left\{
        t\in[0,1]:
        \eqref{eq:modified-scalar-measure-path}
        \text{ admits a smooth solution}
        \right\}
\end{equation*}
is relatively open in \([0,1]\).
\end{proposition}

\begin{proof}
Since \(P_{t_*,\phi_*}\) is formally self-adjoint and elliptic, the Fredholm alternative gives solvability of
\[
        P_{t_*,\phi_*}u=f
\]
for every \(f\in C^{0,\gamma}_0(X,\omega_{\phi_*})\), with \(u\) unique up to addition of a constant. Elliptic regularity then gives \(u\in C^{4,\gamma}(X,\mathbb R)\), and subtracting its \(\omega_{\phi_*}^n\)-average places it in \(C^{4,\gamma}_0(X,\omega_{\phi_*})\). Thus the restriction
\[
        P_{t_*,\phi_*}:
        C^{4,\gamma}_0(X,\omega_{\phi_*})
        \longrightarrow
        C^{0,\gamma}_0(X,\omega_{\phi_*})
\]
is an isomorphism.

To prove openness, set
\[
        \mathcal X_*
        :=
        C^{4,\gamma}_0(X,\omega_{\phi_*}),
        \qquad
        \mathcal Y_*
        :=
        C^{0,\gamma}_0(X,\omega_{\phi_*}).
\]
For \(\psi\in\mathcal X_*\) sufficiently small, define
\[
        \Psi(t,\psi)
        :=
        \frac{
        \mathcal R(t,\phi_*+\psi)
        }{
        \omega_{\phi_*}^n
        }.
\]
Indeed,
\[
\int_X\mathcal R(t,\phi)
=
a_tS_\alpha
+
(1-a_t)\overline S_\alpha V_\alpha
-
S_\alpha
=
0.
\]
Hence \(\Psi(t,\psi)\in Y_*\). Moreover, \(\Psi(t_*,0)
        =
        0\)
and
\[
        D_\psi\Psi(t_*,0)
        =
        P_{t_*,\phi_*}:
        \mathcal X_*
        \longrightarrow
        \mathcal Y_*.
\]
The Banach-space implicit function theorem therefore produces a unique local branch of \(C^{4,\gamma}\)-solutions through
\((t_*,\phi_*)\). Elliptic bootstrapping shows that these solutions are smooth. Hence \(\mathcal T\) is relatively open in \([0,1]\).
\end{proof}


\subsection{Entropy decomposition and uniform entropy bounds}
\label{subsec:entropy-decomposition}

In this subsection we give the uniform entropy bound needed for the a priori estimates in Section~\ref{sec:closedness}. We first recall the entropy term \(\Ent_\omega(\phi)\), which appears in the entropy decomposition of the Mabuchi \(K\)-energy \cite{Chen2000Mabuchi} (see also \cite{Tian1994}).

Let \(\phi\in\mathcal E^1(X,\omega)\). If \(\omega_\phi^n=f\,\omega^n\) for some nonnegative measurable function \(f\), then the entropy term is defined by
\begin{equation*}
        \Ent_\omega(\phi)
        =
        \int_X
        f\log f\,\omega^n
        =
        \int_X
        \log
        \left(
        \frac{\omega_\phi^n}{\omega^n}
        \right)
        \omega_\phi^n.
\end{equation*}
If \(\omega_\phi^n\) is not absolutely continuous with respect to \(\omega^n\), we set
\[
        \Ent_\omega(\phi)
        =
        +\infty.
\]

For a smooth closed real \((1,1)\)-form \(\theta\), the contracted Aubin--Yau energy is defined on \(\Hpot_\omega\) by
\begin{equation}
        E_\omega^\theta(\phi)
        :=
        \frac{1}{n}
        \sum_{j=0}^{n-1}
        \int_X
        \phi\,
        \theta\wedge
        \omega_\phi^j\wedge\omega^{n-1-j},
        \label{eq:contracted-AM-energy}
\end{equation}
which extends \(d_1\)-continuously to \(\mathcal E^1(X,\omega)\) (see \cite[Proposition 4.4]{BDL2017}). 

With all functionals normalized to vanish at the origin, the Mabuchi \(K\)-energy admits the entropy decomposition
\begin{equation}
        \mathcal M_\omega(\phi)
        =
        \Ent_\omega(\phi)
        +
        \overline S_\alpha E_\omega(\phi)
        -
        nE_\omega^{\Ric(\omega)}(\phi).
        \label{eq:Mabuchi-entropy-decomposition}
\end{equation}
For smooth potentials, this identity follows by differentiating both sides and comparing their values at the origin. The right-hand side defines the greatest \(d_1\) lower semicontinuous extension of the Mabuchi \(K\)-energy to \(\mathcal E^1(X,\omega)\) (see \cite[Theorem~4.7]{BDL2017}).

Substituting
\eqref{eq:Mabuchi-entropy-decomposition} into
\eqref{eq:prescribed-functional}, we obtain
\begin{equation}
        \mathcal F_{\omega,\Omega}(\phi)
        =
        \Ent_\omega(\phi)
        -
        nE_\omega^{\Ric(\omega)}(\phi)
        +
        L_{\omega,\Omega}(\phi).
        \label{eq:F-entropy-decomposition}
\end{equation}

Recall the modified path functional \eqref{eq:modified-path-functional}
\[
        \mathcal G_t
        =
        a_t\mathcal M_\omega
        -
        \overline S_\alpha E_\omega
        +
        L_{\omega,\Omega_t}.
\]
Using \eqref{eq:Mabuchi-entropy-decomposition}, we get
\begin{equation*}
        \mathcal G_t(\phi)
        =
        a_t\Ent_\omega(\phi)
        +
        (a_t-1)\overline S_\alpha E_\omega(\phi)        
        -
        na_tE_\omega^{\Ric(\omega)}(\phi)
        +
        L_{\omega,\Omega_t}(\phi).
\end{equation*}
With the normalization
\(E_\omega(\phi)=0\), we then have
\begin{equation}
        \mathcal G_t(\phi)
        =
        a_t\Ent_\omega(\phi)
        -
        na_tE_\omega^{\Ric(\omega)}(\phi)
        +
        L_{\omega,\Omega_t}(\phi).
        \label{eq:path-entropy-normalized}
\end{equation}
The essential point is that \(a_t\geq\varepsilon_0 >0\)
uniformly for \(t\in[0,1]\).

We next establish two auxiliary estimates.
\begin{lemma}
\label{lem:top-J-coercive}
There exist constants \(c>0\) and \(C>0\) such that
\[
        \mathcal J_{\omega,\Omega}(\phi)
        \geq
        c\,d_1(0,\phi)-C
\]
for every \(\phi\in\mathcal E^1_0(X,\omega)\).
\end{lemma}

\begin{proof}
By Yau's theorem \cite{Yau1978}, choose \(\eta\in\alpha\) such that \(\Omega=\overline S_\alpha\eta^n\). Write \(\eta=\omega+\ddbar\psi\) and
\(\widetilde\phi=\phi-\psi\), where \(\psi\) is normalized by \(E_\omega(\psi)=0\). The cocycle identity gives
\[
E_\eta(\widetilde\phi)=E_\omega(\phi)-E_\omega(\psi)=0.
\]

We first relate the \(\mathcal J\)-functional to Chen's top-degree \(J\)-functional in \cite{Chen2018NewPerspective}. With reference metric \(\eta\), one has
\[
        \mathcal J_{\eta,\Omega}(v)
        =
        -\overline S_\alpha E_\eta(v)
        +
        \int_Xv\,\Omega.
\]
Since \(\Omega=\overline S_\alpha\eta^n\), and both
\(\mathcal J_{\eta,\Omega}\) and \(J_{\eta^n}\) are normalized to vanish at the origin, it follows that
\begin{equation*}
        \mathcal J_{\eta,\Omega}(v)
        =
        \overline S_\alpha J_{\eta^n}(v).
        \label{eq:top-J-identification}
\end{equation*}
The cocycle identity gives
\[
        \mathcal J_{\omega,\Omega}(\phi)
        -
        \mathcal J_{\omega,\Omega}(\psi)
        =
        \mathcal J_{\eta,\Omega}(\widetilde\phi).
\]
Consequently,
\begin{equation*}
        \mathcal J_{\omega,\Omega}(\phi)
        =
        \overline S_\alpha
        J_{\eta^n}(\widetilde\phi)
        +
        \mathcal J_{\omega,\Omega}(\psi),
        \label{eq:top-J-reference-change}
\end{equation*}
where the last term is independent of \(\phi\).

We now assume that \(\widetilde\phi\) is smooth. In the notation of Chen, the functional \(J(\phi)\) appearing at the right-hand end of \cite[Proposition~6.1]{Chen2018NewPerspective} is the Aubin \(I\)-functional
\[
        I_\eta(\phi)
        :=
        \int_X
        \phi\,
        \bigl(
        \eta^n-\eta_{\phi}^n
        \bigr).
\]
Hence \cite[Proposition~6.1]{Chen2018NewPerspective} gives
\[
        J_{\eta^n}(\widetilde\phi)
        \geq
        \frac{1}{n+1}
        I_\eta(\widetilde\phi).
\]
Recall that the Aubin--Yau \(J\)-functional is
\[
        J_\eta(\phi)
        :=
        \int_X
        \phi\,\eta^n
        -
        E_\eta(\phi),
\]
and that the standard comparison inequalities give
\[
        \frac{1}{n+1}
        I_\eta(\phi)
        \leq
        J_\eta(\phi)
        \leq
        \frac{n}{n+1}
        I_\eta(\phi).
\]
It follows that
\[
        J_{\eta^n}(\widetilde\phi)
        \geq
        \frac{1}{n}
        J_\eta(\widetilde\phi).
\]
By \cite[equation~(62)]{Darvas2015}, after a harmless volume normalization, there are constants \(c_0>0\) and \(C_0>0\) such that
\[
        J_\eta(\widetilde\phi)
        \geq
        c_0\,d_{1,\eta}(0,\widetilde\phi)-C_0.
\]
Consequently,
\[
        J_{\eta^n}(\widetilde\phi)
        \geq
        c_1\,d_{1,\eta}(0,\widetilde\phi)-C_1.
\]
Since
\[
        d_{1,\eta}(0,\widetilde\phi)
        =
        d_{1,\omega}(\psi,\phi)
        \geq
        d_{1,\omega}(0,\phi)-d_{1,\omega}(0,\psi),
\]
the desired inequality for smooth potentials follows. 

For a general \(\phi\in\mathcal E^1_0(X,\omega)\), take a sequence of smooth normalized potentials converging to \(\phi\) in \(d_1\). The inequality passes to the limit by the \(d_1\)-continuity of \(\mathcal J_{\omega,\Omega}\).
\end{proof}

\begin{lemma}
\label{lem:lower-order-energy-boundedness}
Let \(\theta\) be a smooth closed real \((1,1)\)-form, and let \(\nu\) be a smooth real top-degree form.  For every \(R>0\), there exists a constant \(C_R=C_R(X,\omega,\theta,\nu)>0\)
such that for every \(\phi\in\mathcal E^1_0(X,\omega)\) satisfying \(d_1(0,\phi)\leq R\), we have
\[
        \left|
        E_\omega^\theta(\phi)
        \right|
        +
        \left|
        \int_X\phi\,\nu
        \right|
        \leq
        C_R.
\]

\end{lemma}
\begin{proof}
The boundedness of \(E_\omega^\theta\) on \(d_1\)-bounded subsets follows from Berman--Darvas--Lu \cite[Proposition~4.4]{BDL2017}. 

It remains to control the smooth linear term. Since \(X\) is compact and \(\nu\) is smooth, we may write
\(\nu=f\,\omega^n\) with \(f\in C^\infty(X,\mathbb R)\). Then
\[
        |\nu|\leq \|f\|_{C^0(X)}\,\omega^n .
\]
By \cite[Theorem~5.5]{Darvas2015}, applied with \(\chi(l)=|l|\), for every \(R>0\), there exists
\(C_{\omega,R}>0\) such that
\[
        \int_X|\phi|\,\omega^n
        \leq
        C_{\omega,R}
\]
for every \(\phi\in\mathcal E^1_0(X,\omega)\) with
\(d_1(0,\phi)\leq R\). Therefore, 
\[
        \left|
        \int_X\phi\,\nu
        \right|
        \leq
        \|f\|_{C^0(X)}\int_X|\phi|\,\omega^n
        \leq
       C.
\]
Combining the two estimates gives the desired result.
\end{proof}
\begin{proposition}
\label{prop:path-variational-bounds}
Assume that
\(\mathcal F_{\omega,\Omega}\) is \(d_1\)-coercive on
\(\mathcal E^1_0(X,\omega)\). Let \(I\subset[0,1]\), and suppose that, for every \(t\in I\), there exists a smooth solution \(\phi_t\in\Hpot_\omega\) of
\eqref{eq:modified-scalar-measure-path}, normalized by
\(E_\omega(\phi_t)
        =
        0\).
Then there exists a constant \(C>0\), independent of \(t\in I\), such that
\begin{equation*}
        d_1(0,\phi_t)\leq C, \qquad \text{and} \qquad
         \Ent_\omega(\phi_t)\leq C.
\end{equation*}
\end{proposition}
\begin{proof}
Recall that, by the definition of \(\mathcal J_{\omega,\Omega}\), the modified path functional satisfies
\begin{equation}
        \mathcal G_t(\phi)
        =
        a_t\mathcal F_{\omega,\Omega}(\phi)
        +
        (1-a_t)\mathcal J_{\omega,\Omega}(\phi)
        +
        \int_X\phi\,(\Omega_t-\Omega).
        \label{eq:path-convex-combination}
\end{equation}

By assumption, there exist constants \(\delta_F>0\) and \(C_F>0\) such that
\[
        \mathcal F_{\omega,\Omega}(\phi)
        \geq
        \delta_Fd_1(0,\phi)-C_F
\]
for all \(\phi\in\mathcal E^1_0(X,\omega)\). By Lemma~\ref{lem:top-J-coercive}, there exist
\(\delta_J>0\) and \(C_J>0\) such that
\[
        \mathcal J_{\omega,\Omega}(\phi)
        \geq
        \delta_Jd_1(0,\phi)-C_J.
\]

Since \(\Omega_t-\Omega = (1-t)(\Omega_0-\Omega)\), Proposition~\ref{prop:bounded-target-difference} gives a constant \(B>0\), independent of \(t\in[0,1]\), such that
\[
        \left|
        \int_X\phi\,(\Omega_t-\Omega)
        \right|
        \leq
        B
\]
for every \(\phi\in\mathcal E^1_0(X,\omega)\).

Because \(0<a_t\leq1\), equation \eqref{eq:path-convex-combination} gives
\begin{equation}
        \mathcal G_t(\phi)
        \geq
        \delta\,d_1(0,\phi)-C_0
        \label{eq:uniform-path-coercivity}
\end{equation}
for all \(t\in[0,1]\) and \(\phi\in\mathcal E^1_0(X,\omega)\), where one may take \(\delta =\min \left\{\delta_F,\delta_J \right\}>0\) and \(C_0=\max\left\{C_F,C_J \right\}+B\). For each \(t\in I\), the functional \(\mathcal G_t\) is convex along weak geodesics. Since \(\phi_t\) is a smooth critical point, the same argument as in Lemma~\ref{lem:critical-point-minimizer} shows that \(\phi_t\) is a global minimizer of \(\mathcal G_t\) among smooth potentials. Consequently,
\[
        \mathcal G_t(\phi_t)
        \leq
        \mathcal G_t(0)
        =
        0.
\]
Combining this with \eqref{eq:uniform-path-coercivity} gives
\[
        d_1(0,\phi_t)
        \leq
        \frac{C_0}{\delta},
\]
which gives the \(d_1\)-bound.

It remains to control the entropy. Since \(E_\omega(\phi_t)= 0\), equation \eqref{eq:path-entropy-normalized} gives
\[
        a_t\Ent_\omega(\phi_t)
        =
        \mathcal G_t(\phi_t)
        +
        na_tE_\omega^{\Ric(\omega)}(\phi_t)
        -
         L_{\omega,\Omega_t}(\phi_t).
\]
Using \(\mathcal G_t(\phi_t)\leq 0\), we obtain
\[
        a_t\Ent_\omega(\phi_t)
        \leq
        na_t
        \left|
        E_\omega^{\Ric(\omega)}(\phi_t)
        \right|
        +
        \left|
         L_{\omega,\Omega_t}(\phi_t)
        \right|.
\]

The \(d_1\)-bound and Lemma~\ref{lem:lower-order-energy-boundedness} give a uniform bound for \(E_\omega^{\Ric(\omega)}(\phi_t)\). Since
\[
        L_{\omega,\Omega_t}(\phi_t)
        =
        (1-t)L_{\omega,\Omega_0}(\phi_t)
        +
        tL_{\omega,\Omega}(\phi_t),
\]
Lemma \ref{lem:lower-order-energy-boundedness} again gives a uniform bound for \(L_{\omega,\Omega_t}(\phi_t)\). It follows from \(a_t\geq\varepsilon_0>0\) that
\[
        \Ent_\omega(\phi_t)
        \leq
        C
\]
uniformly for \(t\in I\), completing the proof.
\end{proof}


\subsection{\(d_1\)-coercivity from smooth solvability}
\label{subsec:coercivity-from-solvability}
We now prove that the smooth solvability of the prescribed scalar curvature measure equation \eqref{eq:scalar-measure} forces \(d_1\)-coercivity of \(\mathcal F_{\omega,\Omega}\). The coercivity on \(\Hpot_\omega\) is an application of Chen's result for the higher-degree twisted \(K\)-energy \cite[Section~6]{Chen2018NewPerspective}. The extension from \(\Hpot_\omega\) to \(\mathcal E^1_0(X,\omega)\) uses some approximation results of Berman--Darvas--Lu together with the entropy decomposition recalled in the preceding subsection.

\begin{theorem}\label{thm:existence-implies-coercivity}
Let \(\Omega\in\mathcal V_\alpha^+\). Suppose that \eqref{eq:scalar-measure} admits a smooth solution in
\(\Hpot_\omega\). Then the functional \(\mathcal F_{\omega,\Omega}\) is \(d_1\)-coercive on \(\mathcal E^1_0(X,\omega)\).
\end{theorem}

\begin{proof}
The coercivity of \(\mathcal{F}_{\omega,\Omega}\) on the smooth potential space \(\Hpot_\omega\) follows from \cite[Section~6, proof of the necessary part of Conjecture~6.5]{Chen2018NewPerspective}. In Chen's notation, \(\mathcal F_{\omega,\Omega}\) is the top-degree case \(k=n\) at the midpoint of the twisted path:
\[
        \mathcal F_{\omega,\Omega}
        =
        2\left(
        \frac12\,\mathcal M_\omega
        +
        \frac12\,\mathcal J_{\omega,\Omega}
        \right).
\]
Thus the existence of a smooth solution of \eqref{eq:scalar-measure} gives constants \(\delta>0\) and \(C>0\)
such that for every smooth potential \(\phi\in\Hpot_\omega\) with \(E_\omega(\phi)=0\),
\[
        \mathcal F_{\omega,\Omega}(\phi)
        \geq
        \delta\,d_1(0,\phi)-C.
\]

It remains to extend the coercivity to \(\mathcal E^1_0(X,\omega)\). Let \(\phi\in\mathcal E^1_0(X,\omega)\). If \(\Ent_\omega(\phi)=+\infty\), then the decomposition \eqref{eq:F-entropy-decomposition} together with Lemma \ref{lem:lower-order-energy-boundedness} gives \(\mathcal F_{\omega,\Omega}(\phi)=+\infty\), and the desired result is immediate.

Assume therefore that \(\Ent_\omega(\phi)<+\infty\). Applying \cite[Theorem~3.2]{BDL2017} with \(p=1\) and \(f=0\), there exist smooth potentials \(\phi_j\in\Hpot_\omega\) such that \(\phi_j\rightarrow\phi\) in \(d_1\), and
\[
        \Ent_\omega(\phi_j)
        \longrightarrow
        \Ent_\omega(\phi).
\]
Moreover, by \cite[Proposition~4.4]{BDL2017} and Lemma \ref{lem:smooth-linear-term-d1-continuity}, we have
\[
        E_\omega^{\Ric(\omega)}(\phi_j)
        \longrightarrow
        E_\omega^{\Ric(\omega)}(\phi),
        \qquad
        L_{\omega,\Omega}(\phi_j)
        \longrightarrow
        L_{\omega,\Omega}(\phi).
\]
Hence
\[
        \mathcal F_{\omega,\Omega}(\phi_j)
        \longrightarrow
        \mathcal F_{\omega,\Omega}(\phi).
\]

Since \(E_\omega(\phi_j)\to E_\omega(\phi)=0\), set
\[
        \widehat\phi_j
        :=
        \phi_j-\frac{E_\omega(\phi_j)}{V_\alpha}.
\]
Then
\(E_\omega(\widehat\phi_j)=0\) and \(\widehat\phi_j\to\phi\) in \(d_1\). By Lemma~\ref{lem:constant-invariance},
\[
        \mathcal F_{\omega,\Omega}(\widehat\phi_j)
        =
        \mathcal F_{\omega,\Omega}(\phi_j).
\]
Passing to the limit gives
\[
        \mathcal F_{\omega,\Omega}(\phi)
        \geq
        \delta\,d_1(0,\phi)-C
\]
for every \(\phi\in\mathcal E^1_0(X,\omega)\). Thus \(\mathcal F_{\omega,\Omega}\) is \(d_1\)-coercive on
\(\mathcal E^1_0(X,\omega)\).
\end{proof}

As an immediate consequence, the existence of one positive scalar curvature metric in the class implies \(d_1\)-coercivity for every admissible target, as well as uniform geodesic stability.

\begin{corollary}
\label{cor:psc-implies-stability}
If the K\"ahler class \(\alpha\) contains a PSC K\"ahler metric, then, for every \(\Omega\in\mathcal V_\alpha^+\), the functional \(\mathcal F_{\omega,\Omega}\) is \(d_1\)-coercive on \(\mathcal E^1_0(X,\omega)\). Consequently, the \(\mathcal F\)-functional is uniformly geodesically stable.
\end{corollary}
\begin{proof}
Let \(\omega_0\in\alpha\) with positive scalar curvature. Write \(\omega_0=\omega+\ddbar\phi_0\) for some \(\phi_0\in\Hpot_\omega\), and define \(\Omega_0:=\Scal(\omega_0)\,\omega_0^n\). Then \(\Omega_0\in\vap\), and \(\phi_0\) solves \eqref{eq:scalar-measure} with prescribed measure \(\Omega_0\). Theorem~\ref{thm:existence-implies-coercivity} shows that \(\mathcal F_{\omega,\Omega_0}\) is \(d_1\)-coercive. Corollary~\ref{cor:target-independent-coercivity} gives the same conclusion for every admissible measure, and Corollary~\ref{cor:analytic-PSC-stability-implies-radial-stability} gives uniform geodesic stability.
\end{proof}


\section{A priori estimates and closedness}
\label{sec:closedness}

In this section, we establish uniform estimates for smooth solutions of the modified continuity path \eqref{eq:modified-scalar-measure-path}. The main analytic step is a uniform \(C^0\)-estimate for both the K\"ahler potential and the logarithmic volume ratio. We adapt the method from \cite{DeruelleDiNezza2022} to the specific exponential structure of the coupled system arising from our path. Once this zeroth-order control is established, Chen--Cheng's \cite{ChenCheng2021I} estimates apply to give uniform metric equivalence and higher-order estimates. Together with the uniform entropy bound obtained in Proposition~\ref{prop:path-variational-bounds}, these estimates imply the closedness of the parameter set \(\mathcal T\).


\subsection{The coupled system}
\label{subsec:coupled-system}

Let \(t\in[0,1]\), and suppose that
\(\phi_t\in\Hpot_\omega\) is a smooth solution of
\eqref{eq:modified-scalar-measure-path}. Write
\(\omega_t
        :=
        \omega_{\phi_t}\)
and define the logarithmic volume ratio
\[
        F_t
        :=
        \log
        \frac{\omega_t^n}{\omega^n}.
\]
Since
\(\Omega_t
        =
        (1-t)\Omega_0+t\Omega\)
is a smooth positive volume form satisfying
\[
        \int_X\Omega_t
        =
        S_\alpha
        =
        \overline S_\alpha V_\alpha,
\]
there is a smooth function \(h_t\) such that
\[
        \Omega_t
        =
        \overline S_\alpha e^{h_t}\omega^n.
\]
Because \(\Omega_0\) and \(\Omega\) are fixed smooth positive volume forms, for every integer \(k\geq0\) there exists
\(C_k=C_k(X,\omega,\Omega_0,\Omega,k)>0\) such that
\[
        \sup_{t\in[0,1]}
        \lVert h_t\rVert_{C^k(X,\omega)}
        \leq
        C_k.
\]
Define
\[
        A_t
        :=
        \frac{\overline S_\alpha}{a_t},
        \qquad
        B_t
        :=
        \frac{1-a_t}{a_t}\,
        \overline S_\alpha.
\]
Since \(\varepsilon_0\leq a_t\leq1\), we have the uniform bounds
\[
        \overline S_\alpha
        \leq
        A_t
        \leq
        \frac{\overline S_\alpha}{\varepsilon_0},
        \qquad
        0
        \leq
        B_t
        \leq
        \frac{1-\varepsilon_0}{\varepsilon_0}
        \overline S_\alpha.
\]
Dividing \eqref{eq:modified-scalar-measure-path} by \(\omega_t^n\), we obtain
\[
        \Scal(\omega_t)
        =
        A_t e^{h_t-F_t}
        -
        B_t.
\]
On the other hand,
\[
        \Ric(\omega_t)
        =
        \Ric(\omega)
        -
        \ddbar F_t,
\]
and hence
\[
        \Scal(\omega_t)
        =
        \operatorname{tr}_{\omega_t}\Ric(\omega)
        -
        \Delta_{\omega_t}F_t.
\]
Consequently,
\eqref{eq:modified-scalar-measure-path}
is equivalent to the coupled system
\begin{equation}
\begin{cases} 
\displaystyle \omega_{\phi_t}^n = e^{F_t}\omega^n, \\[6pt] \displaystyle \Delta_{\omega_{\phi_t}}F_t = \operatorname{tr}_{\omega_{\phi_t}}\Ric(\omega) + B_t - A_t e^{h_t-F_t}. 
\end{cases}
\label{eq:modified-coupled-system}
\end{equation}
In particular, the coefficient \(A_t\) of the exponential term has a uniform positive lower bound, and \(A_t\), \(B_t\), and \(h_t\) are uniformly controlled in \(t\).


\subsection{The uniform \(C^0\)-estimate}
\label{subsec:uniform-C0}
In this subsection, we prove the uniform \(C^0\)-estimate for smooth solutions of the system \eqref{eq:modified-coupled-system}.
\begin{proposition}
\label{prop:uniform-C0-estimate}
Let \(t\in[0,1]\), and let \((\phi_t,F_t)\) be a smooth solution of \eqref{eq:modified-coupled-system} normalized by \(\sup_X\phi_t=0\). Assume that
\[
        \Ent_\omega(\phi_t)
        =
        \int_X
        F_t e^{F_t}\omega^n
        \leq
        C_E.
\]
Then there exists a constant \(C=C(X,\omega,\varepsilon_0,\Omega_0,\Omega,C_E)>0\), such that
\[
        \lVert\phi_t\rVert_{C^0(X)}
        +
        \lVert F_t\rVert_{C^0(X)}
        \leq
        C.
\]
\end{proposition}

For the remainder of this subsection, fix \(t\in[0,1]\), suppress the subscript \(t\), and write \(\phi=\phi_t\), \(F=F_t\), \(h=h_t\). We keep the notation \(A_t\) and \(B_t\) for the coefficients in \eqref{eq:modified-coupled-system}. All bounds below depend only on \( X,\omega,\varepsilon_0,\Omega_0,\Omega,C_E\) and are independent of \(t\). Define
\[
        b
        :=
        \frac{1}{V_\alpha}
        \int_X
        e^F\sqrt{1+F^2}\,\omega^n.
\]
Since \(\int_X e^F\omega^n=V_\alpha\), we have \(b\geq1\). We next derive a uniform upper bound for \(b\). Note that \(x\longmapsto e^x\sqrt{1+x^2}\) is increasing and bounded on \((-\infty,1]\), hence
\[
\int_{\{F\le 1 \}} e^F\sqrt{1+F^2}\,\omega^n \le \sqrt{2}e V_\alpha.
\]
Since \((-x)e^x\leq e^{-1}\) for \(x<0\), the entropy bound gives
\begin{align*}
\int_{\{F\geq1\}}Fe^F\omega^n \le &\int_{\{F\geq0\}}Fe^F\omega^n\\
=&  \int_XFe^F\omega^n
        +
        \int_{\{F<0\}}
        (-F)e^F\omega^n  \\
        \le& C_E + \frac{V_\alpha}{e}.
\end{align*}
Consequently,
\begin{align*}
        b
        =
        \frac{1}{V_\alpha}
        \int_X
        e^F\sqrt{1+F^2}\,\omega^n &\le \frac{1}{V_\alpha} \left(\int_{\{F\le 1 \}} e^F\sqrt{1+F^2}\,\omega^n +\int_{\{F\ge 1 \}} e^F\sqrt{1+F^2}\,\omega^n\right) \\
        &\le \sqrt{2}e+\frac{1}{V_\alpha}\int_{\{F\geq1\}}2Fe^F\omega^n\\
        &\le C.
\end{align*}
By Yau's theorem \cite{Yau1978}, there exists a unique smooth potential \(\psi\in\Hpot_\omega\), normalized by \(\sup_X\psi=0\), such that
\begin{equation}
        \omega_\psi^n
        =
        b^{-1}
        e^F\sqrt{1+F^2}\,\omega^n
        \label{eq:auxiliary-MA-equation}
\end{equation}
Indeed, the right-hand side is a smooth positive volume form with total
mass
\[
        b^{-1}
        \int_X
        e^F\sqrt{1+F^2}\,\omega^n
        =
        V_\alpha.
\]

\begin{lemma}
\label{lem:phi-psi-comparison}
Let \(\Lambda>0\) and \(\lambda>0\) be fixed. Then there exists
\( C=C(X,\omega,\varepsilon_0,\Omega_0,\Omega,C_E,\Lambda,\lambda)>0\), such that
\[
        \lambda\psi-\Lambda\phi
        \leq
        C.
\]
\end{lemma}

\begin{proof}
Choose \(0<\delta <\min\left\{1,\frac{\lambda}{2\Lambda} \right\}\). At every point of \(X\), either
\[
        \sqrt{1+F^2}
        \geq
        \frac{2b}{\delta^n}, \qquad\text{or} \qquad         \sqrt{1+F^2}
        <
        \frac{2b}{\delta^n}.
\]
In the first case,
\[
        e^F
        \leq
        \frac12\delta^n
        b^{-1}e^F\sqrt{1+F^2},
\]
while in the second case,
\[
        e^F
        \leq
        e^{\sqrt{1+F^2}}\le 
            \exp
        \left(
        \frac{2b}{\delta^n}
        \right).
\]
Consequently,
\begin{align*}
        \omega_\phi^n
        &=
        e^F\omega^n                                           \\
        &\leq
        \frac12\delta^n
        b^{-1}e^F\sqrt{1+F^2}\,\omega^n
        +
        \exp
        \left(
        \frac{2b}{\delta^n}
        \right)
        \omega^n                                               \\
        &=
        \frac12\delta^n\omega_\psi^n
        +
        \exp
        \left(
        \frac{2b}{\delta^n}
        \right)
        \omega^n.
\end{align*}
Since
\[
        \omega_{\delta\psi}
        =
        (1-\delta)\omega+\delta\omega_\psi
        \geq
        \delta\omega_\psi,
\]
we have
\[
        \omega_{\delta\psi}^n
        \geq
        \delta^n\omega_\psi^n.
\]
Thus
\[
        \omega_\phi^n
        \leq
        \frac12\omega_{\delta\psi}^n
        +
        \exp
        \left(
        \frac{2b}{\delta^n}
        \right)
        \omega^n.
\]
Set \(f_\delta:=\exp\left(\frac{2b}{\delta^n}\right)\). We apply the relative Ko{\l}odziej estimate \cite[Theorem~3.3]{DarvasDiNezzaLu2021} with
\[
        u=\phi,
        \qquad
        \chi=\delta\psi,
        \qquad
        a=\frac12,
        \qquad
        f=f_\delta.
\]
We verify its hypothesis. Here \(\operatorname{Cap}_{v}\) denotes the relative Monge--Amp\`ere capacity. Applying \cite[Proposition~3.10]{DarvasDiNezzaLu2021} to the model potential \(0\), with \(p=2\), gives
\[
        \int_E f_\delta\,\omega^n
        \leq
        C_0V_\alpha^{-2}
        \bigl[
        \operatorname{Cap}_{0}(E)
        \bigr]^2
\]
for every Borel set \(E\subset X\). Moreover, \cite[Lemma~4.2]{DarvasDiNezzaLu2021}, applied with \(\varepsilon=1-\delta\) and \(w=\psi\), gives
\[
        \operatorname{Cap}_{0}(E)
        \leq
        (1-\delta)^{-n}
        \operatorname{Cap}_{\delta\psi}(E).
\]
Consequently,
\[
        \int_E f_\delta\,\omega^n
        \leq
        C_0V_\alpha^{-2}(1-\delta)^{-2n}
        \bigl[
        \operatorname{Cap}_{\delta\psi}(E)
        \bigr]^2.
\]
Since \(\phi\) is smooth, \(P[\phi]=0\), while \(\delta\psi\leq0=P[\phi]\). Thus \(P[\phi]\) is less singular than \(\delta\psi\), as required in \cite[Theorem~3.3]{DarvasDiNezzaLu2021}. Since \(\delta\) is fixed and \(b\leq C\), the \(L^2\)-norm of \(f_\delta\) is uniformly bounded. The relative Ko{\l}odziej estimate therefore yields
\[
        \phi
        \geq
        \delta\psi-C.
\]
Because \(\psi\leq0\) and \(\lambda-\Lambda\delta>0\), we obtain
\begin{align*}
        \lambda\psi-\Lambda\phi
        &\leq
        (\lambda-\Lambda\delta)\psi
        +
        \Lambda C                                          
        \leq
        \Lambda C,
\end{align*}
which completes the proof.
\end{proof}

\begin{lemma}
\label{lem:upper-C0-bounds}
There exists \(C=C(X,\omega,\varepsilon_0,\Omega_0,\Omega,C_E)>0\) such that
\[
        \lVert\phi\rVert_{C^0(X)}
        +
        \lVert\psi\rVert_{C^0(X)}
        +
        \sup_XF
        \leq
        C.
\]
\end{lemma}

\begin{proof}
Choose \(\rho_0>0\) such that \(\Ric(\omega)\geq-\rho_0\,\omega\) and set \(\Lambda:=\rho_0+1\). Recall that Skoda's integrability \cite{Skoda1972} on compact K\"ahler manifolds tells us that there exist constants \(\beta=\beta(X,\omega)>0\) and \(C=C(X,\omega)>0\) such that
\[
        \int_X
        e^{-\beta(u-\sup_Xu)}
        \,\omega^n
        \leq
        C
\]
for every \(u\in\PSH(X,\omega)\) (see \cite[Theorem~8.11]{GuedjZeriahi2017}). Fix \(0<\lambda<\min\left\{1,\beta/4\right\}\), and consider 
\[
H:=F+\lambda\psi-\Lambda\phi.
\] 
We then compute
\begin{align*}
        \Delta_{\omega_\phi}H
        &=
        \operatorname{tr}_{\omega_\phi}\Ric(\omega)
        +
        B_t
        -
        A_t e^{h-F}
        +
        \lambda
        \operatorname{tr}_{\omega_\phi}\omega_\psi          
        -
        \lambda
        \operatorname{tr}_{\omega_\phi}\omega
        -
        n\Lambda
        +
        \Lambda
        \operatorname{tr}_{\omega_\phi}\omega.
\end{align*}
Since \(B_t\geq0\) and \(\Lambda-\rho_0-\lambda=1-\lambda \geq0\), it follows that
\begin{equation}  \label{eq:H-laplacian}
\begin{aligned}
        \Delta_{\omega_\phi}H
        &\geq
        -
        A_t e^{h-F}
        -
        n\Lambda
        +
        \lambda
        \operatorname{tr}_{\omega_\phi}\omega_\psi\\
        & \geq
        -
        A_t e^{h-F}
        -
        n\Lambda
        +
        n\lambda b^{-\frac1n}
        \left(1+F^2\right)^{\frac{1}{2n}}.
\end{aligned}
\end{equation}
where the last inequality is given by the arithmetic--geometric mean inequality.

Let \(x_0\in X\) be a maximum point of \(H\). We first show that \(F(x_0)\) is uniformly bounded from above. If \(F(x_0)\leq0\), then there is nothing to prove.  Suppose that \(F(x_0)>0\). Since \(A_t\) and \(h\) are uniformly
bounded,
\[
        A_t e^{h(x_0)-F(x_0)}
        \leq
        C.
\]
Since
\(\Delta_{\omega_\phi}H(x_0)\leq0\), together with \eqref{eq:H-laplacian} and \(1\leq b\leq C\), we obtain
\[
        \left(1+F(x_0)^2\right)^{\frac{1}{2n}}
        \leq
        C.
\]
Thus, in either case,
\[
        F(x_0)
        \leq
        C.
\]
Next, Lemma~\ref{lem:phi-psi-comparison} gives
\[
        \lambda\psi-\Lambda\phi
        \leq
        C.
\]
Therefore, for every \(x \in X\),
\[
       H(x)\le H(x_0)
        =
        F(x_0)
        +
        \lambda\psi(x_0)
        -
        \Lambda\phi(x_0)
        \leq
        C.
\]
Consequently, since \(\phi\leq0\),
\begin{equation}
        F
        \leq
        C-\lambda\psi+\Lambda\phi
        \leq
        C-\lambda\psi,
        \label{eq:F-upper-by-psi}
\end{equation}

Because \(4\lambda\leq\beta\) and \(\sup_X\psi=0\), Skoda's integrability gives
\begin{equation}
        \int_X
        e^{-4\lambda\psi}\,\omega^n
        \leq
        C.
        \label{eq:Skoda-psi}
\end{equation}
It follows from \eqref{eq:F-upper-by-psi}, using \(\psi\leq0\), that
\[
        \int_Xe^{2F}\omega^n
        \leq
        C\int_Xe^{-2\lambda\psi}\omega^n
        \leq
        C.
\]
Hence \(F \in L^2(X,\omega^n)\). Applying Ko{\l}odziej's \(L^p\)-estimate \cite{Kolodziej1998} to
\[
        \omega_\phi^n
        =
        e^F\omega^n
\]
and using \(\sup_X\phi=0\), we obtain
\[
        \lVert\phi\rVert_{C^0(X)}
        \leq
        C.
\]

We next estimate the density in \eqref{eq:auxiliary-MA-equation}. Since \(b\geq1\),
\[
        \left(
        b^{-1}e^F\sqrt{1+F^2}
        \right)^2
        \leq
        e^{2F}(1+F^2).
\]
The elementary inequality
\[
        e^{2x}(1+x^2)
        \leq
        \frac{1}{2}\bigl(1+e^{4x}\bigr),
        \qquad
        x\in\mathbb R,
\]
together with \eqref{eq:F-upper-by-psi}, gives
\begin{align*}
        \int_X
        \left(
        b^{-1}e^F\sqrt{1+F^2}
        \right)^2
        \omega^n
        &\leq
        \frac{1}{2}
        \int_X
        \bigl(
        1+e^{4F}
        \bigr)
        \omega^n                                             \\
        &\leq
        C
        \int_X
        \bigl(
        1+e^{-4\lambda\psi}
        \bigr)
        \omega^n                                             \\
        &\leq
        C,
\end{align*}
where the last inequality follows from \eqref{eq:Skoda-psi}. The density in \eqref{eq:auxiliary-MA-equation} has total mass \(V_\alpha\) and is uniformly bounded in \(L^2(X,\omega^n)\). Applying Ko{\l}odziej's \(L^p\)-estimate again to \eqref{eq:auxiliary-MA-equation} yields
\[
        \lVert\psi\rVert_{C^0(X)}
        \leq
        C.
\]
Combining with \eqref{eq:F-upper-by-psi}, we conclude that
\[
        \sup_XF
        \leq
        C,
\]
which completes the proof.
\end{proof}

\begin{lemma}\label{lem:F-lower-C0}
Suppose that \(\lVert\phi\rVert_{C^0(X)}\leq C_\phi\). Then there exists \(C=C(X,\omega,\varepsilon_0,\Omega_0,\Omega,C_\phi)>0\) such that
\[
        \inf_XF
        \geq
        -C.
\]
\end{lemma}

\begin{proof}
Choose \(\rho_1>0\) such that \(\Ric(\omega)\leq\rho_1\omega\) and set \(\Lambda_1:=\rho_1+1\). Consider \(K:=F+\Lambda_1\phi\). Using \eqref{eq:modified-coupled-system}, we compute
\begin{align*}
        \Delta_{\omega_\phi}K
        &=
        \operatorname{tr}_{\omega_\phi}\Ric(\omega)
        +
        B_t
        -
        A_t e^{h-F}
        +
        n\Lambda_1
        -
        \Lambda_1
        \operatorname{tr}_{\omega_\phi}\omega              \\
        &\leq
        B_t
        +
        n\Lambda_1
        -
        \operatorname{tr}_{\omega_\phi}\omega
        -
        A_t e^{h-F}                                        \\
        &\leq
        C
        -n\left(\omega^n/\omega_\phi^n\right)^{1/n}-
        A_t e^{h-F}\\
        &\leq C-ne^{-F/n}-
        A_t e^{h-F}\\
        &\le C-
        A_t e^{h-F},
\end{align*}
where we use the arithmetic--geometric mean inequality in the third inequality.

Let \(y_0\in X\) be a minimum point of \(K\). Then
\[
        A_t e^{h(y_0)-F(y_0)}
        \leq C-\Delta_{\omega_\phi}K(y_0)\le C.
\]
Since \(A_t\geq\overline S_\alpha>0\) and \(h\) is uniformly bounded, we obtain
\[
        F(y_0)\geq -C.
\]
Consequently, for every \(x\in X\),
\[
        F(x)+\Lambda_1\phi(x)
        \geq
        F(y_0)+\Lambda_1\phi(y_0) \ge -C,
\]
which completes the proof.
\end{proof}

\begin{proof}[Proof of Proposition~\ref{prop:uniform-C0-estimate}] 
Lemma~\ref{lem:upper-C0-bounds} gives
\[
        \lVert\phi_t\rVert_{C^0(X)}
        +
        \sup_XF_t
        \leq
        C.
\]
Applying Lemma~\ref{lem:F-lower-C0} then yields
\[
        \inf_XF_t
        \geq
        -C.
\]
Therefore,
\[
        \lVert\phi_t\rVert_{C^0(X)}
        +
        \lVert F_t\rVert_{C^0(X)}
        \leq
        C,
\]
uniformly for \(t\in[0,1]\).
\end{proof}


\subsection{Metric equivalence and higher-order estimates}
\label{subsec:higher-order-estimates}

Once the \(C^0\)-bound is obtained, we can apply Chen--Cheng's result \cite{ChenCheng2021I} to derive the uniform higher-order estimates.

\begin{theorem}
\label{thm:uniform-smooth-estimates}
Let \(t\in[0,1]\), and let \((\phi_t,F_t)\) be a smooth solution of \eqref{eq:modified-coupled-system} normalized by \(\sup_X\phi_t=0\). Assume that
\[
        \Ent_\omega(\phi_t)
        =
        \int_X
        F_t e^{F_t}\omega^n
        \leq
        C_E.
\]
Then there exists a constant \(C=C(X,\omega,\varepsilon_0,\Omega,\Omega_0,C_E)>0\) such that
\begin{equation*}
        C^{-1}\omega
        \leq
        \omega_{\phi_t}
        \leq
        C\omega.
        \label{eq:uniform-metric-equivalence}
\end{equation*}
Moreover, for every integer \(k\geq0\), there exists a constant  \(C_k=C(X,\omega,k,\varepsilon_0,\Omega,\Omega_0,C_E)>0\) such that
\begin{equation*}
        \lVert\phi_t\rVert_{C^k(X,\omega)}
        +
        \lVert F_t\rVert_{C^k(X,\omega)}
        \leq
        C_k.
        \label{eq:uniform-Ck-estimates}
\end{equation*}
\end{theorem}
\begin{proof}
Define \(f_t:=A_t e^{h_t-F_t}-B_t\). Then \(f_t=\Scal(\omega_{\phi_t})\). The uniform bounds for \(A_t\), \(B_t\), and \(h_t\), together with the \(C^0\)-bound for \(F_t\) from Proposition~\ref{prop:uniform-C0-estimate}, imply that
\begin{equation}
        \lVert f_t\rVert_{C^0(X)}
        \leq
        C.
        \label{eq:f-t-C0-bound}
\end{equation}
The coupled system can therefore be written as
\begin{equation}
\begin{cases}
        \displaystyle
        \omega_{\phi_t}^n
        =
        e^{F_t}\omega^n,
        \\[5pt]
        \displaystyle
        \Delta_{\omega_{\phi_t}}F_t
        =
        -f_t
        +
        \operatorname{tr}_{\omega_{\phi_t}}\Ric(\omega).
\end{cases}
        \label{eq:Chen-Cheng-form}
\end{equation}
This is the system considered in \cite[Equations~(1.1)--(1.2)]{ChenCheng2021I}, with
\[
        f=f_t,
        \qquad
        \eta=\Ric(\omega).
\]
Here \(\eta\) is fixed and smooth, while \(f_t\) is uniformly bounded in \(C^0\) by \eqref{eq:f-t-C0-bound}. By
\cite[Theorem~4.1]{ChenCheng2021I}, Proposition \ref{prop:uniform-C0-estimate} and \eqref{eq:f-t-C0-bound} imply
\[
        \operatorname{tr}_{\omega}
        \omega_{\phi_t}
        =
        n+\Delta_\omega\phi_t
        \leq
        C.
\]
Together with the uniform bound for \(F_t\), this yields uniform metric equivalence.

We then apply \cite[Theorem~4.3]{ChenCheng2021I}. For every \(1<p<\infty\), it gives
\begin{equation}
        \lVert\phi_t\rVert_{W^{4,p}(X,\omega)}
        +
        \lVert F_t\rVert_{W^{2,p}(X,\omega)}
        \leq
        C_p,
        \label{eq:uniform-W-estimates}
\end{equation} 
where \(C_p\) is independent of \(t\). Choose \(p>2n\), and fix \(0<\gamma<1-\frac{2n}{p}\). By the Sobolev embedding theorem,
\[
        \lVert\phi_t\rVert_{C^{3,\gamma}(X,\omega)}
        +
        \lVert F_t\rVert_{C^{1,\gamma}(X,\omega)}
        \leq
        C.
\]
Since
\[
        f_t=A_t e^{h_t-F_t}-B_t,
\]
where \(A_t\) and \(B_t\) are uniformly bounded constants and \(h_t\) is uniformly bounded in every \(C^m\)-norm, the family \(f_t\) is uniformly bounded in \(C^{1,\gamma}\). The second equation in \eqref{eq:Chen-Cheng-form} is therefore a uniformly elliptic equation with uniformly \(C^{1,\gamma}\)-bounded coefficients and right-hand side. Hence, Schauder estimates give
\[
        \lVert F_t\rVert_{C^{3,\gamma}(X)}
        \leq
        C.
\]
Differentiating the logarithmic form
\[
        \log\frac{\omega_{\phi_t}^n}{\omega^n}=F_t
\]
gives linear uniformly elliptic equations for the first derivatives of \(\phi_t\), with coefficients \((\omega_{\phi_t})^{i\bar j}\). Schauder estimates then improve the regularity of \(\phi_t\). Returning to the second equation improves the regularity of \(F_t\), and hence that of \(f_t\). Repeating this argument yields, for every \(k\geq0\),
\[
        \lVert\phi_t\rVert_{C^k(X,\omega)}
        +
        \lVert F_t\rVert_{C^k(X,\omega)}
        \leq
        C_k,
\]
which completes the proof.
\end{proof}


\subsection{Closedness of the parameter set}

Recall the parameter set
\[
        \mathcal T
        :=
        \left\{
        t\in[0,1]:
        \eqref{eq:modified-scalar-measure-path}
        \text{ admits a smooth solution}
        \right\}.
\]
By construction, \(0\in\mathcal T\), and Proposition~\ref{prop:modified-path-openness} shows that \(\mathcal T\) is relatively open in \([0,1]\).

\begin{proposition}
\label{prop:modified-path-closedness}
Assume that \(\mathcal F_{\omega,\Omega}\) is \(d_1\)-coercive on \(\mathcal E^1_0(X,\omega)\). Then \(\mathcal T\) is closed in \([0,1]\).
\end{proposition}

\begin{proof}
Let \(t_j\in\mathcal T\),  with \(t_j\longrightarrow t_\infty\in[0,1]\). For every \(j\), choose a smooth solution \(\widetilde\phi_j\) of \eqref{eq:modified-scalar-measure-path} at \(t=t_j\), normalized by
\(E_\omega(\widetilde\phi_j)=0\). Proposition~\ref{prop:path-variational-bounds} gives
\[
        \sup_j
        \Ent_\omega(\widetilde\phi_j)
        <\infty.
\]
Set \(\phi_j:= \widetilde\phi_j-\sup_X\widetilde\phi_j\).
Then \(\sup_X\phi_j=0\) and \(\phi_j\) is still a solution at \(t=t_j\) with
\[
        \sup_j
        \Ent_\omega(\phi_j)
        <\infty.
\]
Write \(F_j:= \log\frac{\omega_{\phi_j}^n}{\omega^n}\). Theorem~\ref{thm:uniform-smooth-estimates} gives, for every integer
\(k\geq0\),
\[
        \lVert\phi_j\rVert_{C^k(X,\omega)}
        +
        \lVert F_j\rVert_{C^k(X,\omega)}
        \leq
        C_k,
\]
where \(C_k\) is independent of \(j\).

By the Arzel\`a--Ascoli theorem and a diagonal argument, after passing to a subsequence, there exists \(\phi_\infty\in C^\infty(X,\mathbb R)\) such that \(\phi_j\rightarrow\phi_\infty\) in \(C^\infty(X)\). The uniform metric equivalence in Theorem~\ref{thm:uniform-smooth-estimates} implies \(\omega_{\phi_\infty}\geq C^{-1}\omega\), and hence  \(\phi_\infty\in\Hpot_\omega\).

Since \(\phi_j\to\phi_\infty\) in \(C^\infty\), the scalar curvatures \(\Scal(\omega_{\phi_j})\) and the volume forms \(\omega_{\phi_j}^n\) converge smoothly. Moreover, \(a_{t_j}\to a_{t_\infty}\) and \(\Omega_{t_j}\to\Omega_{t_\infty}\) smoothly on \(X\). Hence we may pass to the limit in \eqref{eq:modified-scalar-measure-path}.
This gives
\[
        \left(
        a_{t_\infty}\Scal(\omega_{\phi_\infty})
        +
        (1-a_{t_\infty})\overline S_\alpha
        \right)
        \omega_{\phi_\infty}^n
        =
        \Omega_{t_\infty}.
\]
Thus \(t_\infty\in\mathcal T\), proving that \(\mathcal T\) is closed.
\end{proof}

We can now prove the existence argument.

\begin{theorem}
\label{thm:coercivity-implies-existence}
Suppose that \(\mathcal F_{\omega,\Omega}\) is \(d_1\)-coercive on \(\mathcal E^1_0(X,\omega)\).
Then the equation \eqref{eq:scalar-measure} admits a unique smooth solution, up to addition of constants, in \(\Hpot_\omega\).
\end{theorem}

\begin{proof}
The set \(\mathcal T\) is nonempty because \(0\in\mathcal T\).
It is relatively open by Proposition~\ref{prop:modified-path-openness} and closed by Proposition~\ref{prop:modified-path-closedness}. Since \([0,1]\) is connected, we have \(\mathcal T=[0,1]\).

In particular, \(1\in\mathcal T\). By construction, \(a_1=1\) and \(\Omega_1=\Omega\). Hence, at \(t=1\),
\eqref{eq:modified-scalar-measure-path} becomes
\[
        \Scal(\omega_\phi)\,\omega_\phi^n
        =
        \Omega.
\]
Uniqueness follows from Corollary~\ref{cor:uniqueness}.
\end{proof}

Combining the above results, we obtain the following variational characterization.

\begin{corollary}
\label{cor:analytic-PSC-stability-and-solvability}
Let \(\alpha\) be a K\"ahler class with positive total scalar curvature, and let \(\omega \in \alpha\) be a reference K\"ahler metric. Then the following are equivalent:
\begin{enumerate}[label=\textup{(\roman*)}]
\item for one, equivalently every, \(\Omega\in\mathcal V_\alpha^+\), the functional \(\mathcal F_{\omega,\Omega}\) is \(d_1\)-coercive on \(\mathcal E^1_0(X,\omega)\);
\item for one, equivalently every,
\(\Omega\in\mathcal V_\alpha^+\), the equation \eqref{eq:scalar-measure} admits a smooth solution in \(\Hpot_\omega\).
\end{enumerate}
Whenever these conditions hold, the solution of \eqref{eq:scalar-measure} is unique up to addition of constants for every \(\Omega\in\mathcal V_\alpha^+\).
\end{corollary}

\begin{proof}
Assume first that \(\mathcal F_{\omega,\Omega_0}\) is \(d_1\)-coercive for some \(\Omega_0\in\mathcal V_\alpha^+\). By Corollary~\ref{cor:target-independent-coercivity}, \(\mathcal F_{\omega,\Omega}\) is \(d_1\)-coercive for every \(\Omega\in\mathcal V_\alpha^+\). Theorem~\ref{thm:coercivity-implies-existence} then gives a smooth solution for every admissible measure. Uniqueness follows from Corollary~\ref{cor:uniqueness}.

Conversely, suppose that the equation admits a smooth solution for some \(\Omega_0\in\mathcal V_\alpha^+\).
Theorem~\ref{thm:existence-implies-coercivity} shows that \(\mathcal F_{\omega,\Omega_0}\) is \(d_1\)-coercive.
\end{proof}


\section{Uniform geodesic stability and \(d_1\)-coercivity}
\label{sec:geodesic-stability}
In this section, we prove that uniform geodesic stability of the \(\mathcal F\)-functional is equivalent to \(d_1\)-coercivity. The proof adapts the compactness and diagonal construction used in the proof of \cite[Theorem~6.5]{BDL2017}: from a sequence violating coercivity, one constructs a finite-energy geodesic ray with controlled radial slope. Fix an admissible measure \(\Omega\in\vap\). We keep the normalization \(\mathcal F_{\omega,\Omega}(0)=0\) in this section without further mention.


\subsection{Finite-energy geodesic rays with controlled radial slope}

We first establish a local lower bound and a compactness property for \(\mathcal F_{\omega,\Omega}\) on \(d_1\)-bounded subsets of \(\mathcal E^1_0(X,\omega)\).

\begin{proposition}
\label{prop:F-bounded-sublevel-compactness}
Let \(R,M>0\). Then the set
\[
        \left\{
        \phi\in\mathcal E^1_0(X,\omega):
        d_1(0,\phi)\leq R,\quad
        \mathcal F_{\omega,\Omega}(\phi)\leq M
        \right\}
\]
is relatively compact in \((\mathcal E^1_0(X,\omega),d_1)\). Moreover, for every \(R>0\), there exists
\(C_R=C_R(X,\omega,\Omega,R)>0\) such that
\begin{equation}
        \mathcal F_{\omega,\Omega}(\phi)
        \geq
        -C_R
        \label{eq:F-lower-bound-on-d1-ball}
\end{equation}
for every \(\phi\in\mathcal E^1_0(X,\omega)\), with \(d_1(0,\phi)\leq R\).

\end{proposition}

\begin{proof}
On the normalized space \(\mathcal E^1_0(X,\omega)\), recall the entropy decomposition
\eqref{eq:F-entropy-decomposition} 
\[
        \mathcal F_{\omega,\Omega}(\phi)
        =
        \Ent_\omega(\phi)
        -
        nE_\omega^{\Ric(\omega)}(\phi)
        +
        L_{\omega,\Omega}(\phi).
\]
By Lemma~\ref{lem:lower-order-energy-boundedness}, there exists \(B_R=B_R(X,\omega,\Omega,R)>0\)
such that for \(\phi\in\mathcal E^1_0(X,\omega)\) with \(d_1(0,\phi)\leq R\),
\[
        \left|
        E_\omega^{\Ric(\omega)}(\phi)
        \right|
        +
        \left|
        L_{\omega,\Omega}(\phi)
        \right|
        \leq
        B_R.
\]
Since \(\Ent_\omega(\phi)\geq0\), the decomposition formula gives
\[
        \mathcal F_{\omega,\Omega}(\phi)
        \geq
        -n\left|E_\omega^{\Ric(\omega)}(\phi)\right|
        -
        \left|L_{\omega,\Omega}(\phi)\right|
        \geq
        -nB_R.
\]
This proves \eqref{eq:F-lower-bound-on-d1-ball}, after setting
\(C_R=nB_R\).

We now prove the compactness. Let \(\{\phi_j\}\subset\mathcal E^1_0(X,\omega)\) satisfy
\[
        d_1(0,\phi_j)\leq R,
        \qquad
        \mathcal F_{\omega,\Omega}(\phi_j)\leq M.
\]
Since \(E_\omega(\phi_j)=0\), Lemma~\ref{lem:lower-order-energy-boundedness} then gives
\[
        \mathcal M_\omega(\phi_j)
        =
        \mathcal F_{\omega,\Omega}(\phi_j)
        -
        L_{\omega,\Omega}(\phi_j)
        \leq
        M+B_R.
\]
Therefore
\cite[Corollary~4.8]{BDL2017} gives a \(d_1\)-convergent subsequence. Since \(E_\omega\) is \(d_1\)-continuous, its limit belongs to \(\mathcal E^1_0(X,\omega)\).
\end{proof}

The following argument shows that if \(\mathcal F_{\omega,\Omega}\) fails to satisfy a linear lower bound with prescribed slope \(\sigma\in\mathbb R\), then there exists a unit-speed finite-energy geodesic ray whose \(\mathcal F_{\omega,\Omega}\)-radial slope is at most \(\sigma\). The proof is a quantitative adaptation of the argument of \cite[Theorem~2.16 and Example~2.17]{BermanBoucksomJonsson2021}.

\begin{proposition}
\label{prop:ray-extraction}
Let \(\sigma\in\mathbb R\). Suppose that there is no constant \(C\in\mathbb R\) such that
\[
        \mathcal F_{\omega,\Omega}(\phi)
        \geq
        \sigma\,d_1(0,\phi)-C
\]
for every \(\phi\in\mathcal E^1_0(X,\omega)\). Then there exists a unit-speed finite-energy \(d_1\)-geodesic ray \(\ell\), with \(\ell_0=0\), such that for every \(t\geq0\),
\[
        \mathcal F_{\omega,\Omega}(\ell_t)
        \leq
        \sigma t.
\]
\end{proposition}
\begin{proof}
By our assumption, for every integer \(j\geq1\), there exists \(\phi_j\in\mathcal E^1_0(X,\omega)\) such that
\[
        \mathcal F_{\omega,\Omega}(\phi_j)
        \leq
        \sigma d_1(0,\phi_j)-j.
\]
Set \(r_j:=d_1(0,\phi_j)\) and \(\sigma_+:=\max\{\sigma,0\}\). We first claim that \(r_j\to+\infty\). Otherwise, after passing to a subsequence, we would have \(r_j\leq R\) for some \(R>0\). It follows from \eqref{eq:F-lower-bound-on-d1-ball} in Proposition~\ref{prop:F-bounded-sublevel-compactness} that
\[
        -C_R\leq\mathcal F_{\omega,\Omega}(\phi_j)\leq \sigma r_j-j\leq
        \sigma_+R-j,
\]
which is impossible for \(j\) sufficiently large.

Let \([0,r_j]\ni t\mapsto\phi_{j,t}\) be the unit-speed finite-energy geodesic segment joining \(0\) to
\(\phi_j\). Since \(E_\omega\) is affine along finite-energy geodesics and \(E_\omega(0)=E_\omega(\phi_j)=0\), the entire segment is contained in \(\mathcal E^1_0(X,\omega)\). By geodesic convexity, for \(0\leq t\leq r_j\),
\begin{align*}
        \mathcal F_{\omega,\Omega}(\phi_{j,t})
        &\leq
        \frac{t}{r_j}\mathcal F_{\omega,\Omega}(\phi_j)\\
        &\leq
        \sigma t-\frac{jt}{r_j}
        \leq
        \sigma t.
\end{align*}

Fix \(T>0\). For all sufficiently large \(j\), \(r_j\geq T\), and for every \(t\in[0,T]\),
\[
        d_1(0,\phi_{j,t})=t\leq T,
        \qquad
        \mathcal F_{\omega,\Omega}(\phi_{j,t})
        \leq
        \sigma t
        \leq
        \sigma_+T.
\]
Thus the images of the curves \([0,T]\ni t\longmapsto\phi_{j,t}\) are contained in
\[
        \left\{
        \phi\in\mathcal E^1_0(X,\omega):
        d_1(0,\phi)\leq T,\quad
        \mathcal F_{\omega,\Omega}(\phi)
        \leq \sigma_+T+1
        \right\},
\]
which is relatively compact by Proposition~\ref{prop:F-bounded-sublevel-compactness}. Since these curves are all \(1\)-Lipschitz, the Arzel\`a--Ascoli theorem, followed by a diagonal argument over \(T\in\mathbb N\), gives a subsequence and a curve \(\ell:[0,+\infty)\rightarrow\mathcal E^1_0(X,\omega)\) such that, for every \(T>0\),
\[
        \sup_{0\leq t\leq T}
        d_1(\phi_{j,t},\ell_t)
        \longrightarrow0.
\]
Since each segment is parameterized with unit speed, passing to the limit gives
\[
        d_1(\ell_s,\ell_t)=|s-t|,
        \qquad s,t\geq0.
\]
In particular, \(\ell_0=0\) and \(\ell\) has unit speed.

It remains to verify that \(\ell\) is a finite-energy geodesic ray. Fix \(T>0\), and let \([0,1]\ni s\mapsto\gamma_s\) be the finite-energy geodesic segment joining \(0\) to \(\ell_T\). The curves \([0,1]\ni s\longmapsto\phi_{j,sT}\) are the finite-energy geodesic segments joining \(0\) to \(\phi_{j,T}\), and
\(d_1(\phi_{j,T},\ell_T)\rightarrow0\). Hence, for every \(s\in[0,1]\), \cite[Proposition~4.3]{BDL2017} gives
\(d_1(\phi_{j,sT},\gamma_s)\rightarrow0\). On the other hand, the locally uniform convergence above gives
\(d_1(\phi_{j,sT},\ell_{sT})\rightarrow0\). Therefore \(\gamma_s=\ell_{sT}\) for \(s\in[0,1]\). Thus the restriction of \(\ell\) to every compact interval is a finite-energy geodesic segment, and hence \(\ell\) is a
finite-energy geodesic ray.

Finally, the \(d_1\)-lower semicontinuity of \(\mathcal F_{\omega,\Omega}\) yields,
for every \(t\geq0\),
\[
        \mathcal F_{\omega,\Omega}(\ell_t)
        \leq
        \liminf_{j\to\infty}\mathcal F_{\omega,\Omega}(\phi_{j,t})
        \leq
        \sigma t.
\]
Since \(t\mapsto\mathcal F_{\omega,\Omega}(\ell_t)\) is convex and \(\mathcal F_{\omega,\Omega}(\ell_0)=0\), its radial slope is well-defined, and
\[
        \mathcal F_{\omega,\Omega}\{\ell\}
        =
        \lim_{t\to\infty}
        \frac{\mathcal F_{\omega,\Omega}(\ell_t)}{t}
        \leq
        \sigma,
\]
which completes the proof.
\end{proof}

\subsection{Completion of the variational characterization}

We now prove the converse implication of Corollary~\ref{cor:analytic-PSC-stability-implies-radial-stability},  namely that uniform geodesic stability implies \(d_1\)-coercivity, and complete the proof of Theorem~\ref{thm:variational-main}.

\begin{theorem}
\label{thm:radial-stability-implies-coercivity}
Suppose that there exists a constant \(\delta_0>0\) such that
\[
        \mathcal F_\alpha\{\ell\}
        \geq
        \delta_0\,v_1(\ell)
\]
for every nonconstant finite-energy \(d_1\)-geodesic ray \(\ell\) starting from \(0\) and parameterized with constant \(d_1\)-speed. Then there exists a constant \(\delta>0\), independent of \(\Omega\), such that for every \(\Omega\in\mathcal V_\alpha^+\) there exists \(C>0\) satisfying
\[
        \mathcal F_{\omega,\Omega}(\phi)
        \geq
        \delta\,d_1(0,\phi)-C
\]
for every \(\phi\in\mathcal E^1_0(X,\omega)\).
\end{theorem}

\begin{proof}
Fix \(\delta:=\delta_0/2\) and let \(\Omega\in\mathcal V_\alpha^+\). We claim that there exists \(C>0\) such that
\[
        \mathcal F_{\omega,\Omega}(\phi)
        \geq
        \delta\,d_1(0,\phi)-C
\]
for every \(\phi\in\mathcal E^1_0(X,\omega)\).

Suppose otherwise. Proposition~\ref{prop:ray-extraction}, applied with \(\sigma=\delta\), gives a unit-speed nonconstant finite-energy \(d_1\)-geodesic ray \(\ell\) with \(\ell(0)=0\) and \( v_1(\ell)=1\) such that
\[
        \mathcal F_{\omega,\Omega}\{\ell\}
        \leq
        \delta.
\]
Therefore
\[
        \delta_0
        \leq
        \mathcal F_\alpha\{\ell\}
        =
        \mathcal F_{\omega,\Omega}\{\ell\}
        \leq
        \delta
        =
        \frac{\delta_0}{2},
\]
which is impossible.

Hence there exists a constant \(C=C(\delta,\Omega)>0\) such that
\[
        \mathcal F_{\omega,\Omega}(\phi)
        \geq
        \delta\,d_1(0,\phi)-C
\]
for every \(\phi\in\mathcal E^1_0(X,\omega)\). Since \(\Omega\in\mathcal V_\alpha^+\) was arbitrary, this proves \(d_1\)-coercivity for every admissible target.
\end{proof}

We can now complete the proof of Theorem~\ref{thm:variational-main}.
\begin{proof}[Proof of Theorem~\ref{thm:variational-main}]
The equivalence \(\textup{(i)}\iff \textup{(ii)}\) is clear. Indeed, a solution for a positive admissible target has positive scalar curvature, while a positive scalar curvature metric \(\omega_*\in\alpha\) solves the equation with \(\Omega_*=\Scal(\omega_*)\,\omega_*^n\in\mathcal V_\alpha^+\). The equivalences  \(\textup{(ii)}\iff\textup{(iii)}\iff\textup{(iv)}\)
follow from Corollary~\ref{cor:analytic-PSC-stability-and-solvability}. The equivalence \(\textup{(iv)}\iff\textup{(v)}\) is obtained by Corollary~\ref{cor:analytic-PSC-stability-implies-radial-stability}, and Theorem~\ref{thm:radial-stability-implies-coercivity}. Finally, the uniqueness of solutions follows from Corollary~\ref{cor:uniqueness}.
\end{proof}

We finally prove Corollary~\ref{cor:global-parametrization-psc}.

\begin{proof}[Proof of Corollary~\ref{cor:global-parametrization-psc}]
We use the \(C^\infty\)-topology on both spaces.  Fix any reference K\"ahler metric \(\omega\in\alpha\), we identify \(\mathcal H_\omega^{\mathrm{psc}}/\mathbb R\) with the slice
\[
        \mathcal H_{\omega,0}^{\mathrm{psc}}
        :=
        \left\{
        \phi\in\mathcal H_\omega^{\mathrm{psc}}:
        E_\omega(\phi)=0
        \right\}.
\]

First, note that the map
\[
        \mathscr S_\omega:
        \mathcal H_\omega^{\mathrm{psc}}/\mathbb R
        \longrightarrow
        \mathcal V_\alpha^+,
        \qquad
        [\phi]\longmapsto
        \Scal(\omega_\phi)\,\omega_\phi^n
\]
is injective by the uniqueness statement in Theorem~\ref{thm:variational-main}. It is surjective because, under the assumption that \(\alpha\) contains a positive scalar curvature K\"ahler metric, Theorem~\ref{thm:variational-main} implies that for every \(\Omega\in\mathcal V_\alpha^+\) the equation
\[
        \Scal(\omega_\phi)\,\omega_\phi^n=\Omega
\]
admits a smooth solution.  Hence \(\mathscr S_\omega\) is bijective.

The map \(\mathscr S_\omega\) is continuous in the \(C^\infty\)-topology, since scalar curvature and the volume form depend smoothly on the potential. It remains to see that its inverse is continuous. Fix \(\Omega_*\in\mathcal V_\alpha^+\), and let \(\phi_*\in\mathcal H_{\omega,0}^{\mathrm{psc}}\) be the unique normalized solution of
\[
        \mathscr S_\omega(\phi_*)=\Omega_*.
\]
For \(k\geq0\) and \(0<\gamma<1\), consider the Banach manifolds
\[
\begin{aligned}
        \mathscr H_{k,\gamma}
        &:=
        \left\{
        \phi\in C^{k+4,\gamma}(X,\mathbb R):
        \omega_\phi>0,\quad E_\omega(\phi)=0
        \right\},                                          \\
        \mathscr A_{k,\gamma}
        &:=
        \left\{
        \Theta\in
        C^{k,\gamma}
        \bigl(X,\Lambda_{\mathbb R}^{n,n}\bigr):
        \int_X\Theta=S_\alpha
        \right\}.
\end{aligned}
\]
By \eqref{eq:AY-variation}, the tangent space of \(\mathscr H_{k,\gamma}\) at \(\phi_*\) is
\[
        T_{\phi_*}\mathscr H_{k,\gamma}
        =
        C^{k+4,\gamma}_0(X,\omega_{\phi_*}),
\]
Similarly,
\[
\begin{aligned}
        T_{\Omega_*}\mathscr A_{k,\gamma}
        =
        \left\{
        \Theta:
        \int_X\Theta=0
        \right\}                                           
        =
        \left\{
        f\,\omega_{\phi_*}^n:
        f\in C^{k,\gamma}_0(X,\omega_{\phi_*})
        \right\}.
\end{aligned}
\]
Under these identifications, the linearization is
\[
        D\mathscr S_\omega(\phi_*)[u]
        =
        P_{\omega_{\phi_*}}u\,\omega_{\phi_*}^n.
\]

Proposition~\ref{prop:automatic-nondegeneracy} gives the isomorphism for \(k=0\). Since \(P_{\omega_{\phi_*}}\) is a fourth-order elliptic operator with smooth coefficients, elliptic regularity gives an isomorphism
\[
        P_{\omega_{\phi_*}}:
        C^{k+4,\gamma}_0(X,\omega_{\phi_*})
        \longrightarrow
        C^{k,\gamma}_0(X,\omega_{\phi_*})
\]
for every \(k\geq0\). The Banach-space implicit function theorem therefore gives neighborhoods
\[
        \phi_*\in U_{k,\gamma}\subset\mathscr H_{k,\gamma},
        \qquad
        \Omega_*\in V_{k,\gamma}\subset\mathscr A_{k,\gamma},
\]
and a smooth map
\[
        \Psi_{k,\gamma}:
        V_{k,\gamma}
        \longrightarrow
        U_{k,\gamma} \qquad \text{with} \qquad \mathscr S_\omega
        \bigl(\Psi_{k,\gamma}(\Omega)\bigr)
        =
        \Omega.
\]
After shrinking \(V_{k,\gamma}\), we may assume that all its elements are positive volume forms. If \(\Omega\in V_{k,\gamma}\) is smooth, then \(\Psi_{k,\gamma}(\Omega)\) is smooth by elliptic bootstrapping. Hence, by the uniqueness of smooth normalized solutions, \(\Psi_{k,\gamma}(\Omega)\) agrees with the globally defined inverse \(\mathscr S_\omega^{-1}(\Omega)\).
Now suppose that \(\Omega_j\to\Omega_*\) in \(C^\infty\), and write \(\phi_j:=\mathscr S_\omega^{-1}(\Omega_j)\).
For every fixed \(k\), we have \(\Omega_j\in V_{k,\gamma}\) for all sufficiently large \(j\), and
therefore
\[
        \phi_j
        =
        \Psi_{k,\gamma}(\Omega_j)
        \longrightarrow
        \Psi_{k,\gamma}(\Omega_*)
        =
        \phi_*
\]
in \(C^{k+4,\gamma}\). Since \(k\) is arbitrary, \(\phi_j\to\phi_*\) in \(C^\infty\). Thus \(\mathscr S_\omega^{-1}\) is continuous, and \(\mathscr S_\omega\) is a homeomorphism.

Finally, \(\mathcal V_\alpha^+\) is convex: if \(\Omega_0,\Omega_1\in\mathcal V_\alpha^+\), then
\[
        (1-t)\Omega_0+t\Omega_1
        \in\mathcal V_\alpha^+,
        \qquad t\in[0,1].
\]
Hence \(\mathcal V_\alpha^+\) is contractible. The homeomorphism \(\mathscr S_\omega\) therefore implies that \(\mathcal H_\omega^{\mathrm{psc}}/\mathbb R\) is contractible.
\end{proof}

\section{PSC K\"ahler metrics on toric manifolds}
\label{sec:toric-PSC}

In this section, we study the existence of PSC K\"ahler metrics on toric manifolds. 


\subsection{The Abreu equation and relative \(\widehat K\)-polystability}
\label{subsec:toric-Abreu}

Let \(X\) be a compact smooth toric K\"ahler manifold of complex dimension \(n\), equipped with an action of the compact torus \(T^n\), and let \(\alpha\in\mathcal K_X\) be any K\"ahler class. Averaging a K\"ahler form in \(\alpha\) over \(T^n\) preserves both positivity and its cohomology class. We may therefore fix a \(T^n\)-invariant K\"ahler form \(\omega\in\alpha\). After choosing a normalization of its moment map, let \(\Delta\subset\mathbb R^n\) be the associated Delzant polytope, defined by
\[
        \Delta
        =
        \left\{
        x\in\mathbb R^n:
        l_m(x)\geq0,\quad 1\leq m\leq N
        \right\},
\]
where
\[
        l_m(x)
        =
        \langle x,\nu_m\rangle-\lambda_m
\]
and \(\nu_m\in\mathbb Z^n\) is the primitive inward-pointing normal to the facet \(F_m=\{l_m=0\}\).
        
Let \(d\mu\) denote Euclidean Lebesgue measure on \(\Delta\). If \(dA_m\) denotes Euclidean hypersurface measure on \(F_m\), define the boundary measure \(d\sigma\) by
\[
        d\sigma|_{F_m}
        =
        \frac{1}{|\nu_m|}\,dA_m.
\]
This is the boundary-measure normalization used in the Abreu integration formula below.

We adopt the convention of \cite{LiLianSheng2023} for the Guillemin potential:
\[
        v_\Delta
        :=
        \sum_{m=1}^N l_m\log l_m.
\]
We use the convention that \(C^\infty(\overline\Delta)\) consists of the restrictions to \(\overline\Delta\) of smooth functions defined on an open neighborhood of \(\overline\Delta\).  Set
\[
        \mathcal S(\Delta)
        :=
        \left\{
        u\in C^0(\overline\Delta):
        D^2u>0\text{ on }\Delta^\circ \text{ and}\
        u-v_\Delta\in C^\infty(\overline\Delta)
        \right\}.
\]
The standard Guillemin--Abreu formalism identifies \(\mathcal S(\Delta)\), modulo affine functions, with the space of smooth \(T^n\)-invariant K\"ahler metrics in the class \(\alpha\); see \cite{Guillemin1994,Abreu1998}.

For \(u\in\mathcal S(\Delta)\), let \((u^{ij})=(u_{ij})^{-1}\). The scalar curvature of the corresponding metric is given on \(\Delta^\circ\) by the Abreu formula
\[
        S(u)
        =
        -
        \sum_{i,j=1}^n
        \frac{\partial^2u^{ij}}
        {\partial x_i\partial x_j}.
\]
Consequently, for a prescribed function \(A\in C^\infty(\overline\Delta)\), the toric prescribed scalar curvature equation is
\[
        S(u)=A,
        \qquad
        u\in\mathcal S(\Delta).
\]
We set a linear functional associated with \(A\) by
\[
        \mathcal L_A(w)
        :=
        \int_{\partial\Delta}w\,d\sigma
        -
        \int_\Delta Aw\,d\mu.
\]
Combining Abreu's scalar curvature formula with Donaldson's integration-by-parts identity \cite[Lemma~3.3.5]{Donaldson2002}, one obtains
\[
\int_\Delta u^{ij}w_{ij}\,d\mu
=
\int_{\partial\Delta}w\,d\sigma
-
\int_\Delta S(u)w\,d\mu.
\]
Consequently, if \(S(u)=A\), then for every affine function \(\ell\), we have
\[
\mathcal L_A(\ell)=0.
\]

Following \cite[Definition~2.5 and Remark~2.6]{LiLianSheng2023}, the pair \((\Delta,A)\) is said to be \emph{relatively \(\widehat K\)-polystable} if
\[
        \mathcal L_A(w)\geq0
\]
for every finite lower semicontinuous convex function \(w:\overline\Delta\rightarrow\mathbb R\), and
\[
        \mathcal L_A(w)=0
        \quad\Longleftrightarrow\quad
        w\ \text{is affine}.
\]

We shall use the following theorem.

\begin{theorem}[Li--Lian--Sheng {\cite[Theorem~1.3]
{LiLianSheng2023}}]
\label{thm:Li-Lian-Sheng}
Let \(\Delta\) be the Delzant polytope associated with a compact smooth toric K\"ahler manifold and a fixed toric K\"ahler class, and let \(A\in C^\infty(\overline\Delta)\). Then the following conditions are equivalent:
\begin{enumerate}[label=\textup{(\roman*)}]
\item
the pair \((\Delta,A)\) is relatively \(\widehat K\)-polystable;

\item
there exists \(u\in\mathcal S(\Delta)\) satisfying
\(S(u)=A\).
\end{enumerate}
\end{theorem}


\subsection{A positive stable density}
\label{subsec:positive-stable-density}
Fix \(b\in\Delta^\circ\).  For a finite lower semicontinuous convex function \(w:\overline\Delta\to\mathbb R\), its subdifferential at \(b\) is
\[
\partial w(b)
:=
\left\{
p\in\mathbb R^n:
w(x)\geq w(b)+\langle p,x-b\rangle
\ \text{for all }x\in\overline\Delta
\right\}.
\]
We say that \(w\) is \emph{normalized at \(b\)} if
\(w(b)=0\) and \(0\in\partial w(b)\).
Equivalently, \(w(b)=\min_{\overline\Delta}w=0\), and hence \(w\geq0\) on \(\overline\Delta\). Denote by \(\mathcal C_b(\Delta)\) the space of such normalized convex functions, and write
\[
\mathcal B_\sigma(w)
:=
\int_{\partial\Delta}w\,d\sigma.
\]

We first record the following estimate, which applies to an arbitrary compact convex polytope. This is a standard interior estimate for normalized convex functions. For the reader's convenience and in order to match our notation, we include a self-contained proof. Related estimates on the toric manifolds appear in \cite[Lemmas~5.1.3, 5.2.3, and 5.2.4]{Donaldson2002}.

\begin{lemma}
\label{lem:normalized-convex-interior-control}
Let \(\Delta\subset\mathbb R^n\) be a compact convex polytope with nonempty interior, and suppose that the restriction of \(d\sigma\) to each facet is a positive constant multiple of Euclidean hypersurface measure.  There exists \(C_1=C(\Delta,b,d\sigma)>0\) such that every \(w\in\mathcal C_b(\Delta)\) satisfies
\begin{equation*}
        \int_\Delta w\,d\mu
        \leq
        C_1\,\mathcal B_\sigma(w).
        \label{eq:convex-interior-L1}
\end{equation*}
Moreover, for every compact set \(K\subset\Delta^\circ\), there exists \(C_2=C(\Delta,b,d\sigma,K)>0\)
such that
\begin{equation*}
        \lVert w\rVert_{C^0(K)}
        +
        \operatorname{Lip}_K(w)
        \leq
        C_2\mathcal B_\sigma(w).
        \label{eq:convex-interior-C0-Lip}
\end{equation*}
\end{lemma}

\begin{proof}
For each facet \(F\) of \(\Delta\), write
\[
P_F:=\operatorname{Conv}(b,F)=\{(1-t)b+ty:y\in F,0\le t\le 1\}.
\]
The pyramids \(P_F\) cover \(\Delta\), and their pairwise intersections have Lebesgue measure zero outside lower-dimensional subsets. Since \(w(b)=0\), convexity gives
\[
w\bigl((1-t)b+ty\bigr)
\leq
t\,w(y).
\]

Let \(dA_F\) denote Euclidean hypersurface measure on \(F\), and let \(d_F>0\) be the Euclidean distance from \(b\) to the affine hyperplane containing \(F\).  In the preceding parametrization, the Euclidean volume element is
\[
d_F\,t^{n-1}\,dt\,dA_F(y).
\]
It follows that
\[
\int_{P_F}w\,d\mu
\leq
d_F\int_F\int_0^1
t^n w(y)\,dt\,dA_F(y) 
=
\frac{d_F}{n+1}\int_F w\,dA_F.
\]
Since \(d\sigma|_F\) is a positive constant multiple of \(dA_F\), summing over the finitely many facets yields
\[
\int_\Delta w\,d\mu
\leq
C\,\mathcal B_\sigma(w),
\]
which proves the first estimate.

By replacing \(K\) by its convex hull, we may assume that \(K\) is convex. Choose \(r>0\) such that
\[
K_{2r}
:=
\left\{
x\in\mathbb R^n:
\operatorname{dist}(x,K)\leq2r
\right\}
\Subset\Delta^\circ,
\]
and define \(K_r\) similarly.  For every \(x\in K_r\), one has \(B_r(x)\subset K_{2r}\). Jensen's inequality and
\eqref{eq:convex-interior-L1} give
\[
w(x)
\leq
\frac{1}{\mu(B_r)}
\int_{B_r(x)}w\,d\mu 
\leq
\frac{1}{\mu(B_r)}
\int_\Delta w\,d\mu 
\leq
C_K\,\mathcal B_\sigma(w).
\]
Hence
\[
\sup_{K_r}w
\leq
C_K\,\mathcal B_\sigma(w).
\]

Fix \(x\in K\), \(p\in\partial w(x)\), and a unit vector \(e\in\mathbb R^n\).  Since \(x\pm re\in K_r\), the subgradient inequality at \(x\) gives
\[
w(x\pm re)
\geq
w(x) \pm r\langle p,e\rangle.
\]
Therefore,
\[
\frac{w(x)-w(x-re)}{r}
\leq
\langle p,e\rangle
\leq
\frac{w(x+re)-w(x)}{r}.
\]
Since \(w\geq0\), it follows that
\[
|\langle p,e\rangle|
\leq
\frac{\sup_{K_r}w}{r}
\leq
C_K\,\mathcal B_\sigma(w).
\]
Taking the supremum over all unit vectors \(e\) yields
\[
|p|
\leq
C_K\,\mathcal B_\sigma(w)
\]
for every \(x\in K\) and every \(p\in\partial w(x)\).

Finally, let \(x,y\in K\), and choose \(p_x\in\partial w(x)\) and \(p_y\in\partial w(y)\).  The subgradient again inequalities imply
\[
\langle p_x,y-x\rangle
\leq
w(y)-w(x)
\leq
\langle p_y,y-x\rangle.
\]
Consequently,
\[
|w(y)-w(x)|
\leq
C_K\,\mathcal B_\sigma(w)\,|x-y|.
\]
Together with the preceding \(C^0\)-estimate, this proves the second estimate.
\end{proof}

\begin{proposition}
\label{lem:positive-density}
Let \(\Delta\subset\mathbb R^n\) be a compact convex polytope with nonempty interior, and suppose that the restriction of \(d\sigma\) to each facet is a positive constant multiple of Euclidean hypersurface measure.  Let
\[
b
:=
\fint_{\partial\Delta} x\,d\sigma
\]
be the barycenter of the boundary measure \(d\sigma\).  Then \(b\in\Delta^\circ\).  Moreover, there exist
\(f\in C^\infty(\overline\Delta)\), with \(f>0\) on \(\overline\Delta\), and \(\delta>0\) such that:
\begin{enumerate}[label=\textup{(\roman*)}]
\item for every affine function \(\ell\),
\begin{equation}
\int_\Delta \ell f\,d\mu
=
\int_{\partial\Delta}\ell\,d\sigma;
\label{eq:positive-density-affine-moments}
\end{equation}

\item for every \(w\in\mathcal C_b(\Delta)\),
\begin{equation}
\mathcal L_f(w)
\geq
\delta\,\mathcal B_\sigma(w).
\label{eq:positive-density-stability}
\end{equation}
\end{enumerate}
Consequently, the pair \((\Delta,f)\) is relatively
\(\widehat K\)-polystable.
\end{proposition}

\begin{proof}
We first show that \(b\in\Delta^\circ\). Let \(F_m\) be any facet, and let \(\ell_m\) be an affine defining function satisfying
\[
        \ell_m\geq0
        \quad\text{on }\Delta,
        \qquad
        F_m=\{\ell_m=0\}\cap\Delta.
\]
The function \(\ell_m\) is strictly positive on the relative interior of every facet distinct from \(F_m\). Since \(d\sigma\) restricts to a positive multiple of hypersurface measure on every facet,
\[
        \int_{\partial\Delta}\ell_m\,d\sigma
        >
        0.
\]
Since \(\ell_m\) is affine,
\[
       \ell_m(b)
        =
        \fint_{\partial\Delta}\ell_m\,d\sigma
        >
        0.
\]
This holds for every facet \(F_m\), and hence \(b\in\Delta^\circ\).

Fix \(s\in(0,1)\), and consider \(T_s(x):=b+s(x-b)\in \Delta^\circ\). Define the finite positive Borel measure
\[
\beta_s:=(T_s)_*(d\sigma)
\]
on \(\Delta\). Then
\(\operatorname{supp}\beta_s
=
T_s(\partial\Delta)
\Subset
\Delta^\circ\).

We first observe that \(\beta_s\) and \(d\sigma\) agree on affine functions.  Indeed, if \(\ell\) is affine, then
\[
\ell(T_s(x))
=
(1-s)\ell(b)+s\ell(x).
\]
Since \(b\) is the barycenter of \(d\sigma\),
\[
\ell(b)\int_{\partial\Delta}d\sigma
=
\int_{\partial\Delta}\ell\,d\sigma.
\]
Therefore,
\[
\begin{aligned}
\int_\Delta\ell\,d\beta_s=
\int_{\partial\Delta}\ell(T_s(x))\,d\sigma(x) =
(1-s)\ell(b)\int_{\partial\Delta}d\sigma
+
s\int_{\partial\Delta}\ell\,d\sigma =
\int_{\partial\Delta}\ell\,d\sigma.
\end{aligned}
\]
Thus \(\beta_s\) has the same mass and barycenter as \(d\sigma\).

Let \(w\in\mathcal C_b(\Delta)\).  Since \(w(b)=0\), for every \(x\in\partial\Delta\), convexity gives
\[
w(T_s(x))
=
w\bigl((1-s)b+sx\bigr)
\leq
s\,w(x).
\]
Therefore,
\begin{align}\label{eq:toric-stable}
\mathcal B_\sigma(w)
-
\int_\Delta w\,d\beta_s
=
\int_{\partial\Delta}
\bigl(
w(x)-w(T_s(x))
\bigr)\,d\sigma(x) 
\geq
(1-s)\mathcal B_\sigma(w).
\end{align}

We next regularize \(\beta_s\) by convolution with a standard nonnegative mollifier \(\rho\in C_c^\infty(\mathbb R^n)\) of unit mass. For \(r>0\), let
\(\rho_r(z):=r^{-n}\rho(z/r)\), and \(h_r:=\rho_r*\beta_s\). For all sufficiently small \(r>0\), \(\operatorname{supp}h_r \Subset \Delta^\circ\).
Moreover, \(h_r\in C_c^\infty(\Delta^\circ)\), with \(h_r\geq0\).

Since \(\rho\) is even, its first moment vanishes.
Hence, for every affine function \(\ell\),
\[
\begin{aligned}
\int_\Delta \ell h_r\,d\mu
&=
\int_{\operatorname{supp}\beta_s}
\int_{\mathbb R^n}
\ell(y+z)\rho_r(z)\,d\mu(z)\,d\beta_s(y) \\
&=
\int_{\operatorname{supp}\beta_s}
\ell(y)\,d\beta_s(y) \\
&=
\int_{\partial\Delta}\ell\,d\sigma.
\end{aligned}
\]

Choose \(r_0>0\) and a compact set \(K'\Subset\Delta^\circ\) such that
\[
\operatorname{supp}\beta_s+\overline{B_{r_0}(0)}
\subset K'.
\]
By Lemma~\ref{lem:normalized-convex-interior-control}, for every \(w\in\mathcal C_b(\Delta)\) we have
\[
\operatorname{Lip}_{K'}(w)
        \leq
        C_2\mathcal B_\sigma(w).
\]
Thus, for \(0<r<r_0\),
\[
\begin{aligned}
\left|
\int_\Delta wh_r\,d\mu
-
\int_\Delta w\,d\beta_s
\right|
&\leq
\int_{\operatorname{supp}\beta_s}
\int_{\mathbb R^n}
|w(y+z)-w(y)|
\rho_r(z)\,d\mu(z)\,d\beta_s(y) \\
&\leq
\operatorname{Lip}_{K^\prime}(w)
\left(
\int_{\mathbb R^n}|z|\rho_r(z)\,d\mu(z)
\right)
\beta_s(\Delta) \\
&\leq
Cr\,\mathcal B_\sigma(w).
\end{aligned}
\]
Choose \(r>0\) sufficiently small that \(2Cr\leq1-s\). Combining with \eqref{eq:toric-stable}, we obtain
\begin{equation}
\mathcal{L}_{h_r}(w)=\mathcal B_\sigma(w)
-
\int_\Delta h_rw\,d\mu
\geq
\gamma\,\mathcal B_\sigma(w),
\qquad
\gamma:=\frac{1-s}{2}>0.
\label{eq:smoothed-density-stability}
\end{equation}
Consequently, \(h_r \ge 0\) and satisfies \eqref{eq:positive-density-affine-moments} and \eqref{eq:positive-density-stability}. 
However, since \(\operatorname{supp}h_r\Subset\Delta^\circ\),
\(h_r\) vanishes in a neighborhood of \(\partial\Delta\) and hence is not strictly positive on \(\overline\Delta\). We next perturb \(h_r\) to obtain a strictly positive density while preserving \eqref{eq:positive-density-affine-moments} and \eqref{eq:positive-density-stability}.

The open set \(U=\{h_r>0\}\) is nonempty. We first choose affinely independent points \(p_0,\ldots,p_n\in U\). We then choose pairwise disjoint balls
\[
B_{\varepsilon_j}(p_j)\Subset U,
\qquad
0\leq j\leq n,
\]
and nonnegative functions \(\chi_j\in C_c^\infty\left(B_{\varepsilon_j}(p_j)\right)\) such that
\[
\int_\Delta\chi_j\,d\mu=1.
\]
Set \(e_0:=1\), and \(e_i(x):=x_i\) for \(1\leq i\leq n\).
Define
\[
\mathbf A_{ij}
:=
\int_\Delta e_i\chi_j\,d\mu,
\qquad
0\leq i,j\leq n.
\]
As the radii \(\varepsilon_j\) tend to zero, \(\mathbf A_{ij}\rightarrow e_i(p_j)\).
The matrix \(\left(e_i(p_j)\right)_{0\leq i,j\leq n}\)
is invertible because \(p_0,\ldots,p_n\) are affinely independent. We may therefore choose the radii sufficiently small that \(\mathbf A\) is invertible.

Let
\[
        m_i
        :=
        \int_\Delta e_i\,d\mu,
        \qquad
        0\leq i\leq n.
\]
For \(t>0\), we define
\[
        f_t
        :=
        h_r
        +
        t
        +
        \sum_{j=0}^n c_j(t)\chi_j, \qquad  c(t)
        :=
        -t\mathbf A^{-1}m.
\]
For \(0\leq i\leq n\), we then have
\begin{align*}
        \int_\Delta e_if_t\,d\mu
        &=
        \int_\Delta e_ih_r\,d\mu
        +
        tm_i
        +
        \sum_{j=0}^n
        \mathbf A_{ij}c_j(t)                                \\
        &=
        \int_\Delta e_ih_r\,d\mu.
\end{align*}
Thus \(f_t\) has the same affine moments as \(h_r\), and hence satisfies \eqref{eq:positive-density-affine-moments}.

We first verify that \(f_t\) is strictly positive. Write
\[
K_\chi
:=
\bigcup_{j=0}^n\operatorname{supp}\chi_j.
\]
Since \(K_\chi\Subset U=\{h_r>0\}\), compactness gives
\[
a
:=
\min_{K_\chi}h_r
>
0.
\]
Since \(c_j(t)=O(t)\) and the functions \(\chi_j\) are fixed, there exists \(C>0\) such that
\[
\left|
\sum_{j=0}^n c_j(t)\chi_j
\right|
\leq
Ct
\qquad
\text{on }\overline\Delta.
\]
Hence, for all sufficiently small \(t>0\),
\[
f_t
\geq
a+t-Ct
>
0
\qquad
\text{on }K_\chi.
\]
On \(\overline\Delta\setminus K_\chi\), all the functions \(\chi_j\) vanish, and therefore
\[
f_t
=
h_r+t
\geq
t
>
0.
\]
Thus we have
\(f_t>0\) on \(\overline\Delta\).

We then verify \eqref{eq:positive-density-stability}. For \(w\in\mathcal C_b(\Delta)\), Lemma
\ref{lem:normalized-convex-interior-control} gives
\[
\int_\Delta w\,d\mu
+
\sum_{j=0}^n
\int_\Delta\chi_jw\,d\mu
\leq
C\,\mathcal B_\sigma(w).
\]
Using \eqref{eq:smoothed-density-stability} and \(\lvert c_j(t)\rvert\leq Ct\), we obtain
\[
\begin{aligned}
\mathcal L_{f_t}(w)
&=
\mathcal L_{h_r}(w)
-
t\int_\Delta w\,d\mu
-
\sum_{j=0}^n
c_j(t)\int_\Delta\chi_jw\,d\mu \\
&\geq
\gamma\,\mathcal B_\sigma(w)
-
t\int_\Delta w\,d\mu
-
\sum_{j=0}^n
|c_j(t)|
\int_\Delta\chi_jw\,d\mu \\
&\geq
\bigl(\gamma-Ct\bigr)\mathcal B_\sigma(w).
\end{aligned}
\]
Choosing \(t>0\) sufficiently small, we obtain
\[
\mathcal L_{f_t}(w)
\geq
\frac{\gamma}{2}\mathcal B_\sigma(w).
\]
Therefore \eqref{eq:positive-density-stability} holds with \(f:=f_t\) and \(\delta:=\frac{\gamma}{2}\).

It remains to verify relative \(\widehat K\) polystability of \((\Delta,f)\).  Let \(u:\overline\Delta\rightarrow\mathbb R\) be a finite, lower-semicontinuous, convex function. Since
\(b\in\Delta^\circ\), the subdifferential \(\partial u(b)\) is nonempty.  Choose \(p\in\partial u(b)\), and define
\[
\ell(x)
:=
u(b)+\langle p,x-b\rangle,
\qquad
w:=u-\ell.
\]
Then \(w(b)=0\), \(0\in\partial w(b)\), and hence \(w\in\mathcal C_b(\Delta)\).  Since \(\mathcal L_f(\ell)=0\), by \eqref{eq:positive-density-affine-moments}, we have
\[
\mathcal L_f(u)
=
\mathcal L_f(w).
\]
Therefore, \eqref{eq:positive-density-stability} gives
\[
\mathcal L_f(u)
\geq
\delta\,\mathcal B_\sigma(w)
\geq
0.
\]

Suppose now that \(\mathcal L_f(u)=0\). Then \(\mathcal B_\sigma(w)=0\). By \eqref{eq:convex-interior-L1},
\[
        0
        \leq
        \int_\Delta w\,d\mu
        \leq
        C_1\mathcal B_\sigma(w)
        =
        0.
\]
Thus \(w=0\) almost everywhere on \(\Delta\). Since a finite convex function is continuous on \(\Delta^\circ\), it follows that \(w\equiv0\) on \(\Delta^\circ\). For \(y\in\partial\Delta\), choose \(y_k\in\Delta^\circ\) with \(y_k\to y\). Lower semicontinuity gives
\[
        w(y)
        \leq
        \liminf_{k\to\infty}w(y_k)
        =
        0.
\]
Since \(w\geq0\), we conclude that \(w(y)=0\). Hence
\(w\equiv0\) on \(\overline\Delta\), and therefore \(u=\ell\) is affine.

Conversely, \eqref{eq:positive-density-affine-moments} gives \(\mathcal L_f(\ell)=0\) for every affine function \(\ell\). Thus \((\Delta,f)\) is relatively \(\widehat K\)-polystable.
\end{proof}

We can now complete the proof of Theorem~\ref{thm:toric-main} and Corollary \ref{cor:toric-PSC-conjectures}.

\begin{proof}[Proof of Theorem~\ref{thm:toric-main}]
Let \(\alpha\in\mathcal K_X\) be any K\"ahler class. Choose a torus-invariant K\"ahler metric in \(\alpha\),
and let \(\Delta\) be its Delzant polytope.
By Proposition~\ref{lem:positive-density}, there exists
\(f\in C^\infty(\overline\Delta)\), with \(f>0\),
such that \((\Delta,f)\) is relatively \(\widehat K\)-polystable.
Theorem~\ref{thm:Li-Lian-Sheng} therefore gives a smooth
torus-invariant K\"ahler metric \(\omega_*\in\alpha\)
whose scalar curvature, viewed as a function on the moment polytope, is \(f\). Hence \(\omega_*\) has positive scalar curvature. Since \(\alpha\) was arbitrary, every K\"ahler class on \(X\) contains
a torus-invariant K\"ahler metric of positive scalar curvature.

In particular, \(S_\alpha>0\). Applying Theorem~\ref{thm:variational-main}, we conclude that for every \(\Omega\in\mathcal V_\alpha^+\), the prescribed scalar curvature measure equation \eqref{eq:scalar-measure} admits a smooth solution in \(\Hpot_\omega\), unique up to addition of constants.
\end{proof}

\begin{proof}[Proof of Corollary~\ref{cor:toric-PSC-conjectures}]
Theorem~\ref{thm:toric-main} gives \(\mathcal K_X\subseteq\mathcal K_X^{\mathrm{psc}}\).
Hence,
\[
        \mathcal K_X^{\mathrm{psc}}
        =
        \mathcal T_X^+
        =
        \mathcal K_X.
\]
Thus Question~\ref{q:class-level-PSC} has an affirmative answer for \(X\).

\end{proof}

\bibliographystyle{alpha}
\bibliography{PSC-variational-approach.bib}
\bigskip
  \footnotesize

  Zehao Sha, \textsc{Institute for Mathematics and Fundamental Physics, Hefei, China}\par\nopagebreak
  Email address: \texttt{zhsha@imfp.org.cn}\par\nopagebreak
  Homepage: \url{https://ricciflow19.github.io/}

\end{document}